\documentclass[10pt,reqno]{amsart}
\usepackage{epsfig}
\usepackage{mathtools}
\usepackage{amssymb}
\usepackage{amsmath}
\usepackage{amsthm}
\usepackage{aliascnt}
\usepackage{indentfirst}
\usepackage{footnpag}
\usepackage{bm}
\usepackage{graphicx,pgf}
\usepackage{tikz}

\allowdisplaybreaks
\numberwithin{equation}{section}

\newtheorem{theorem}{Theorem}[section]
\newtheorem{lemma}[theorem]{Lemma}
\newtheorem{*theorem}[theorem]{$\bullet$ Theorem}
\newtheorem{proposition}[theorem]{Proposition}

\newtheorem{corollary}[theorem]{Corollary}
\newtheorem{*corollary}[theorem]{$\bullet$ Corollary}

\newtheorem*{definition*}{Definition}

\newtheorem{remark}[theorem]{Remark}

\newcounter{tmp}
\newcommand{\norm}[1]{\lVert#1\rVert}
\newcommand{\abs}[1]{\lvert#1\rvert}

\title[Blow-up solutions on tori]{{\bf Mean field equations on tori: \\existence and uniqueness of evenly symmetric blow-up solutions}}
\author[D. Bartolucci]{Daniele Bartolucci$^1$}
\address{$^1$ Department of Mathematics, University
of Rome {\it "Tor Vergata"}, Via della ricerca scientifica n.1, 00133 Roma, Italy}
\email{bartoluc@mat.uniroma2.it}
\author[C. Gui]{Changfeng Gui$^2$}
\address{$^2$Department of Mathematics, The University of Texas at San Antonio, San Antonio, TX 78249, USA.}
\email{\tt changfeng.gui@utsa.edu}
\author[Y. Hu]{Yeyao Hu$^3$}
 \address{$^3$Department of Mathematics,
The University of Texas at San Antonio, San Antonio, TX 78249, USA.}
\email{\tt huyeyao@gmail.com}
\author[A. Jevnikar]{Aleks Jevnikar$^4$}
\address{$^4$Scuola Normale Superiore, Piazza dei Cavalieri 3, 56126 Pisa, Italy}
\email{\tt aleks.jevnikar@sns.it}
\author[W. Yang]{Wen Yang$^5$}
\address{$^5$Division of Mathematics, Wuhan Institute of Physics and Mathematics, Chinese Academy of Science, Wuhan, Hubei, 430071, PRC.}
\email{\tt wyang@wipm.ac.cn}
\date{}
\begin{document}
\thanks{ D.Bartolucci is  partially supported by MIUR Excellence Department Project awarded to the Department of Mathematics, Univ. of Rome Tor Vergata, CUP E83C18000100006; C.Gui and Y. Hu are partially supported by NSF grant DMS-1601885.}

\maketitle

\begin{quote}
{ \textbf{Abstract}}: We are concerned with the blow-up analysis of mean field equations. It has been proven in \cite{bjly} that solutions blowing-up at the same non-degenerate blow-up set are unique. On the other hand, the authors in \cite{bgp} show that solutions with a degenerate blow-up set are in general non-unique. In this paper we first prove that evenly symmetric solutions on a flat torus with a degenerate two-point blow-up set are unique. In the second part of the paper we complete the analysis by proving the existence of such blow-up solutions by using a Lyapunov-Schmidt reduction method. Moreover, we deduce that all evenly symmetric blow-up solutions come from one-point blow-up solutions of the mean field equation on a ``half" torus.


\end{quote}

\section{Introduction}
In this paper, we consider the mean field equation on a flat torus $T:=\mathbb{C}/\mathbb{Z}\omega_1+\mathbb{Z}\omega_2$, i.e.,
\begin{equation}
\label{1.1}
\Delta u + \rho \left(\frac{e^{u}}{\int_{T}e^u}-\frac{1}{|T|}\right)=0,
\end{equation}
\begin{equation}
\label{integnorm}
\int_T e^u = 1,
\end{equation}
where 
$\rho$ is a real parameter, $\mbox{Im}\ \frac{\omega_2}{\omega_1}>0$ and $|T|$ denotes the total area of the torus. For convenience, in this paper we always assume that $|T|=1$.

In the past decade there has been an extensive study of the mean field equation on a general compact Riemann surface $M$ without boundary:
\begin{equation}
\label{1.2}
\Delta_{\scriptscriptstyle M} u + \rho\left(\frac{h e^u}{\int_M h e^u}-\frac{1}{\abs{M}}\right)=0,
\end{equation}
where $\Delta_{\scriptscriptstyle M}$ denotes the corresponding Laplace-Beltrami operator on $(M,g)$, $h\in C^{\infty}(M)$ is a non-negative potential function and $\abs{M}$ is the total area of the surface $M$. To simplify our notation, we shall always assume $\abs{M}=1.$ Equation \eqref{1.2} and its counterpart on bounded planar domains arise in several areas of mathematics and physics and there are by now many results concerning existence (\cite{B1,BMal,BDeM,BDeM2,BdMM,cama,cl2,cl3,DJLW,dj}), uniqueness of solutions (\cite{bgjm,bl,BLin3,BLT,CCL,gm,GM2,GM3,lin2,suz}) and blow-up analysis (\cite{bcct,bjl,bt,bt2,bm,cl,cl4,l,ls}). On one hand, they are derived as a mean field limit in the statistical mechanics description of two dimensional turbulent Euler flows (\cite{clmp,clmp2}) and selfgravitating systems (\cite{KLB,k, wo}). On the other hand, \eqref{1.2} is related to conformal metrics on surfaces with or without conical singularities (\cite{kw,tr}) and to gauge field theories (\cite{y}) possibly coupled with Einstein's general relativity (\cite{cgs,pt,t3}). Recently they have attracted a lot of attention from the analytical point of view due to the close connection to the Chern-Simons-Higgs theory. The relativistic Abelian Chern-Simons gauge field theory was proposed by Jackiw and Weinberg \cite{jw} and Hong et al. \cite{hkp} independently to investigate the physics of high temperature super-conductivity. The energy minimizers of these models satisfy self-dual equations while the Bogomol'nyi-type system of first-order differential equations could be reduced to a single second-order elliptic equation:
\begin{equation}
\label{1.3}
\Delta u + \frac{1}{\epsilon^2}e^u (1-e^u)=4\pi\sum\limits_{i=1}^N \delta_{q_i}\quad\mathrm{in}\quad\mathbb{R}^2,
\end{equation}
where $\delta_{q_i}$ denotes the dirac measure at $q_i$. Equation \eqref{1.3} can be considered on flat tori or on the entire $\mathbb{R}^2$. Tarantello in \cite{t} showed that one type of solutions to \eqref{1.3} converge to the solution of a mean field equation of type \eqref{1.2} after subtracting $2\log{\epsilon}$ and a combination of the Green's function at the singular source $q_j$ when the Chern-Simons coupling constant $\epsilon$ tends to $0$. Latterly, Lin and Yan in \cite{ly} proved the local uniqueness of the blow-up solutions to (\ref{1.3}). More recently that argument has been used by Bartolucci, Lee together with our third and fourth authors in \cite{bjly} to establish the local uniqueness of the blow-up solutions to (\ref{1.2}).\\ To state our main result we need some definitions first.  Let $h(x)$ be a non-negative smooth function which vanishes only at a finite number of points and let $\vec{p}=(p_1,\cdots,p_m)\in M^m$ be such that
$$\{p_1,\cdots,p_m\}\cap\{x\in M\mid h(x)=0\}=\emptyset.$$
We set
\begin{equation}
\label{1.8}
G^{\ast}_i(x)=8\pi R(x,p_i)+8\pi \sum\limits_{j\neq i} G(x,p_j),\quad i=1,\cdots,m,
\end{equation}
where $G(x,y)$ is the Green's function:
\begin{equation*}
-\Delta_{\scriptscriptstyle M} G(x,y) = \delta_y -1 \textrm{ in }M, \quad  \int_{M} G(x,y)dH^2(y)=0,
\end{equation*}
and $R(x,y)$ denotes its regular part. We define
\begin{equation}
\label{1.9}
l(\vec{p})=\sum\limits_{i=1}^m \left[\Delta_{\scriptscriptstyle M} \log{h}(p_i)+8\pi m - K(p_i)\right]h(p_i) e^{G^{\ast}_j(p_i)},
\end{equation}
where $K(x)$ stands for the Gaussian curvature at $x\in M$. Next, we will denote by $V^M(q,r)$ the pre image of the Euclidean ball of radius $r$, $B(q,r)\subset\mathbb{R}^2$, in a suitably defined isothermal coordinates system. For the case $m\geq2$ we fix a sufficiently small constant $r_0\in(0,\frac12)$ and a family of open sets $M_j$ satisfying $M_l\cap M_i=\emptyset$ if $l\neq i$, $\bigcup_{i=1}^m\overline{M}_j=M$, $V^M(p_i,2r_0)\subset M_i,~i=1,\cdots,m$. Then let us set
\begin{equation}
\label{1.11}
D(\vec{p})=\lim_{r\to0}\sum_{i=1}^mh(p_i)e^{G_i^*(p_i)}\left(\int_{M_i\setminus V^M(p_i,r_i)}e^{\Phi_i(x,\vec{p})}dH^2(x)-\frac{\pi}{r_i^2}\right),
\end{equation}
where $M_1=M$ if $m=1$, $r_i=r\sqrt{8h(p_i)e^{G^{\ast}_i(p_i)}}$ and
\begin{equation}
\label{1.12}
\Phi_i(x,\vec{p})=\sum\limits_{l=1}^m 8\pi G(x,p_l) - G^{\ast}_{i}(p_i)+\log{\left(\frac{h(x)}{h(p_i)}\right)}.
\end{equation}
For $(x_1,\cdots,x_m)\in M\times\cdots\times M$ we define
\begin{equation}
\label{1.13}
f_m(x_1,\cdots,x_m)=\sum\limits_{i=1}^m \left[\log{(h(x_i))}+4\pi R(x_i,x_i)\right] +4\pi \sum\limits_{i\neq j} G(x_i,x_j).
\end{equation}

\noindent
If a sequence of solutions of (\ref{1.3}) is not uniformly bounded from above, then it is well known that (see \cite{l}), passing to a subsequence if necessary, it holds,
$$
\rho_n\dfrac{e^{u_n}}{\int\limits_{M}e^{u_n}}\rightharpoonup 8\pi\sum\limits_{i=1}^{m}\delta_{p_i},\quad
\rho_n \to 8\pi m,\;\mbox{ as }n\to +\infty,
$$
weakly in the sense of measures in $M$,  for some $m\in\mathbb{N}$. The points $\{p_1,\cdots,p_m\}$ are said to be the blow-up points (\cite{bm}). From \cite{cl,mw} we know that the blow-up points are critical points of $f_m(x_1,\cdots,x_m)$. Then, Bartolucci et al. \cite{bjly} proved the following theorem.

\begingroup
\setcounter{tmp}{\value{theorem}}
\setcounter{theorem}{0}
\renewcommand{\thetheorem}{\Alph{theorem}}
\begin{theorem}\emph{(\cite{bjly})}\label{tha}
Let $u^{(1)}_n$ and $u^{(2)}_n$ be two sequences of solutions to (\ref{1.2}) with $\rho^{(1)}_n=\rho^{(2)}_n=\rho_n$ and blowing-up at the points $p_j$, for $j=1,\cdots,m$, where $\vec{p}=(p_1,\cdots,p_m)$ is a non-degenerate critical point of $f_m$, i.e.
\begin{equation} \label{non-deg}
det(D^2_M f_m(\vec{p}))\neq0.
\end{equation}
Assume that either,
\begin{itemize}
\item[(1)] $l(\vec{p})\neq 0$, or,
\item[(2)] $l(\vec{p})=0$ and $D(\vec{p})\neq 0$.
\end{itemize}
Then there exists an integer constant $N_0$ sufficiently large such that $u^{(1)}_n=u^{(2)}_n$ for all $n\geq N_0$.
\end{theorem}
\endgroup

A natural question is whether the assumptions of the latter theorem are necessary or not. It turns out that if we drop the non-degeneracy condition \eqref{non-deg} the uniqueness property does not hold anymore in general, as the authors in \cite{bgp} exhibit multiple one-peak solutions blowing-up at a degenerate critical point of $f_1$ on a bounded domain. On the contrary, we will prove that evenly symmetric solutions on a flat torus with $h\equiv 1$ and with a degenerate two-point blow-up set are unique.
\setcounter{theorem}{\thetmp}
\begin{theorem}\label{th2}
Let $u^{(1)}_n$ and $u^{(2)}_n$ be two sequences of solutions to (\ref{1.1})-(\ref{integnorm}) with $\rho^{(1)}_n=\rho^{(2)}_n=\rho_n$ blowing-up at $p_1=\vec{0}$ and at $p_2=\frac{\omega_1}{2}$ or $\frac{\omega_2}{2}$ or $\frac{\omega_1+\omega_2}{2}$. Assume that $u^{(i)}_n$ is evenly symmetric, i.e. $u^{(i)}_n(z)=u^{(i)}_n(-z)$ for all $n$ and $i=1,2$. Then, there exists an $N_0$ sufficiently large such that $u^{(1)}_n=u^{(2)}_n$ for all $n\geq N_0$.
\end{theorem}
We point out that in the latter setting we have $l(\vec{p})=32\pi e^{G^{\ast}_1(p_1)}\neq 0$. On the other hand, the blow-up set $(p_1,p_2)$ is a degenerate critical point of $f_2$ defined in \eqref{1.13} due to the translation invariance. The proof of the uniqueness result follows the one in \cite{bjly} by taking advantage of the evenly symmetric property to bypass the non-degeneracy assumption. More precisely, assuming by contradiction the existence of two distinct blow-up solutions $u^{(i)}_n$ of (\ref{1.3}) we consider their normalized difference
$$
\xi_n=\frac{u^{(1)}_n-u^{(2)}_n}{\norm{u^{(1)}_n-u^{(2)}_n}_{L^{\infty}(M)}}\,.
$$
The starting point in analyzing $\xi_n$ relies on the description of the blow-up solutions carried out by Chen and Lin in \cite{cl}. Moreover, we exploit the evenly symmetric property to deduce an estimate on the distance between the local maximum point and the blow-up point. The latter estimate will be crucially used in all the forthcoming arguments. Next, one can show that after a suitable scaling, $\xi_n$ converges to an entire solution $\xi(x)$ of the linearized problem associated to the Liouville equation:
\begin{equation}
\label{1.4}
\Delta v + e^v=0\quad\textrm{in}\quad \mathbb{R}^2.
\end{equation}
Solutions of (\ref{1.4}) with finite mass are completely classified by Chen and Li \cite{cli} and take the following form:
\begin{equation}
\label{1.5}
v(z) = v_{\mu,a}(z)=\log{\left(\frac{8 e^{\mu}}{(1+ e^{\mu}|z-a|^2)^2}\right)},\quad \mu\in\mathbb{R},\quad a=(a_1,a_2)\in\mathbb{R}^2.
\end{equation}
Baraket and Pacard in \cite{bp} showed that the kernel of the linearized operator at $v_{0,0}$
\begin{equation} \label{lin}
L(\phi)= \Delta \phi+ \frac{8}{(1+|z|^2)^2}\phi
\end{equation}
is spanned by three functions:
$$
	\varphi_0(z)= \frac{1-|z|^2}{|z|^2+1}=\frac{\partial v_{\mu,a}}{\partial\mu}\bigg\rvert_{(\mu,a)=(0,0)},
$$
$$	
	\varphi_1(z)=\frac{z_1}{|z|^2+1}=-\frac{1}{4}\frac{\partial v_{\mu,a}}{\partial a_1}\bigg\rvert_{(\mu,a)=(0,0)}, \quad
\varphi_2(z)=\frac{z_2}{|z|^2+1}=-\frac{1}{4}\frac{\partial v_{\mu,a}}{\partial a_2}\bigg\rvert_{(\mu,a)=(0,0)}.
$$
Thus, we have
\begin{equation} \label{1.6}
\xi(z)=\sum\limits_{i=0}^2 b_i \varphi_i(z)
\end{equation}
for some constants $b_i\in\mathbb{R}$. The idea is then to use suitable Pohozaev identities to prove that $b_i=0$ for each $i$. In particular, the evenly symmetric property is crucially used to guarantee that the elements of the kernel corresponding to the translation invariance vanishes. Finally, after showing that $\xi\not\equiv0$ one gets a contradiction and thus necessarily $u^{(1)}_n\equiv u^{(2)}_n$.

\medskip

Let us conclude this part by giving some comments on the recent study of the local uniqueness property. It turns out that one can also derive such property for the spike solution of Schr\"odinger equation. In \cite{w}, Wei showed that the single interior spike solution of a singularly perturbed semilinear Neumann problem is locally unique at a non-degenerate peak point. Stimulated by the works of Wei \cite{w} and Cao, Noussair and Yan \cite{cny}, various authors have contributed many papers to this subject, see, e.g., \cite{cll,dly,g,glw,gpy,w2}.

\medskip

In the second part of the paper we complete the analysis by proving the existence of such evenly symmetric blow-up solutions using a Lyapunov-Schmidt reduction method.
\begin{theorem}\label{th1}
Let $\epsilon\in (0,\epsilon_0)$ for some $\epsilon_0>0$ small enough and let $\rho=16\pi+\epsilon$. Let $p_1=\vec{0}$ and $p_2=\frac{\omega_1}{2}$ or $\frac{\omega_2}{2}$ or $\frac{\omega_1+\omega_2}{2}$. Then, for each $\epsilon$ there exist a $\lambda>0$ and a solution $u_{\lambda}$ to equation (\ref{1.1}) such that
$$\epsilon= \left(32\pi+o(1)\right)\lambda e^{-\lambda},$$
$$u_{\lambda}(p_i)\rightarrow +\infty \textrm{ for }i=1,2, \quad u_{\lambda}(x)\rightarrow -\infty \textrm{ for all }x\in T\setminus\{p_1,p_2\}$$
as $\epsilon\rightarrow 0$, and
\[u_{\lambda}(z)=u_{\lambda}(-z).\]
Moreover, we have
\[\frac{\rho}{\int_T e^{u_{\lambda}}} e^{u_{\lambda}}\rightarrow 8\pi (\delta_{p_1}+\delta_{p_2}) \textrm{ in a sense of measure, as } \epsilon\rightarrow 0.\]
\end{theorem}

\medskip

Motivated by the computation of the topological degree, Chen and Lin in \cite{cl2} constructed blowing-up solutions under the assumption on $f_m$ in \eqref{1.13} being a Morse function (see also \cite{ef} for a generalization of the latter result under a weaker assumption on the critical points of $f_m$ being ``stable"). However, as already pointed out the function $f_m$ is not a Morse function in our setting. The proof of Theorem~\ref{th1} follows the strategy introduced by Cheng,  the second  and third author of the present paper in \cite{cgh} where the square torus is considered. In particular, we will extend the latter argument to general flat tori. We start by constructing an approximate blowing-up solution to \eqref{1.1}. Then, we study the solvability of the linearized operator in a suitable functional setting. Finally, we reduce the problem to the one-dimensional problem of finding the appropriate scale of the bubbles.

\medskip

Based on this ``local uniqueness" property, we can further show that the evenly symmetric two-point blow-up solutions are one-point blow-up solutions of the mean field equation on a ``half" torus.
It is not too difficult to see that $v_{\lambda}(x)=u_{\lambda}(x+p_2)$ is also a solution of \eqref{1.1} which also blows-up at $p_1,p_2$ ($p_1=0$ and $p_2$ is one of the half periods). By Theorem \ref{th2}, we get
\begin{equation*}
u_\lambda(z)=u_\lambda(z+p_2).
\end{equation*}
In particular, taking $p_2=\frac{\omega_1}{2}$, the solution we build becomes the solution to \eqref{1.1} on a flat torus $T_{\frac12}:=\mathbb{R}^2/\mathbb{Z}\frac{\omega_1}{2}+\mathbb{Z}\omega_2$ which blows-up at the origin. Similar phenomena can be observed for other choices of $p_2$. Let us formulate the above conclusions in the following result:
\begin{corollary}\label{cr1}
All evenly symmetric two-point blow-up solutions of \eqref{1.1} form one-point blow-up solutions of the mean field equation on a ``half" torus.
\end{corollary}

\medskip

The paper is organized as follows. In Section 2 we revisit some a priori estimates of blow-up solutions proved in \cite{cl} by Chen and Lin and we present an estimate on the distance between the local maximum point and the blow-up point. In Section 3 we provide the proof of the uniqueness property stated in Theorem~\ref{th2}. Finally, in Section 4 we construct blowing-up solutions and prove Theorem~\ref{th1}. 


\medskip
\section{Preliminaries}\label{sec2}
In this section we recall some a priori estimates obtained by Chen and Lin in \cite{cl} for blow-up solutions of (\ref{1.1}). Suppose that $u_n$ is a sequence of blow-up solutions of (\ref{1.1})-(\ref{integnorm}) which blow up at $p_1$ and $p_2$, i.e.
\begin{equation}
\label{2.1}
\Delta u_n + \rho_n \left(e^{u_n}-1\right)=0 \textrm{ in }T,\quad \int_{T}e^{u_n}=1,
\end{equation}
where $\rho_n\rightarrow 16\pi$ as $n\rightarrow \infty$. Let
\begin{equation}
\label{2.2}
\lambda_n = \max\limits_{T} u_n,
\end{equation}
and
\begin{equation}
\label{2.3}
\lambda_{n,i}=\max\limits_{B(p_i,\delta)} u_n= u_n(x_{n,i})~ \textrm{ for }~i=1,2,
\end{equation}
where $\delta>0$ is a small fixed constant and $B(p_i,\delta)$ denotes a geodesic ball of radius $\delta$ on $T$ centered at $p_i$. We recall that $M_1$ and $M_2$ are two open sets dividing $T$ into two disjoint parts and $p_i\in M_i$ for $i=1,2$. Furthermore, $r_0$ is chosen as right after (\ref{1.9}) to guarantee that $B(p_i,2r_0)\subset M_i$ for $i=1,2$.

\begin{remark} \label{notation}
To simplify our notation, since $T$ is a flat torus we shall treat $x\in T$ as a point in $\mathbb{R}^2$. Then, the notation $B(p,\delta)$ stands for the set of points $x\in T$ with $d(x,p)<\delta$, where the metric is the one inherited from the Euclidean metric of $\mathbb{R}^2$, i.e.
$$d(x,y):=\min_{z\in\{z\mid z=y+\mathbb{Z}\omega_1+\mathbb{Z}\omega_2\}}|x-z|.$$
\end{remark}

Let us introduce the Green's function $G(x,y)$ of $-\Delta$ on $T$,
\begin{equation}
\label{2.4}
-\Delta_x G(x,y)=\delta_{y}(x) - 1~\textrm{ in }~T,\qquad
\int_{T}G(x,y)dx=0~\textrm{ for all }~y\in T.
\end{equation}
In particular, we have the explicit formula of $G(x,y)$ in terms of doubly periodic functions (see \cite{co}):
\begin{equation}
\label{2.5}
\begin{aligned}
G(x,y)=G(z):=~&\mbox{Im}\left(\frac{|z|^2-\bar{\omega}_1z^2/\omega_1}{2(\omega_1\bar{\omega}_2-\bar{\omega_1}\omega_2)}
-\frac{z}{2\omega_1}+\frac{\omega_2}{12\omega_1}\right)
-\frac{1}{2\pi}\bigg\vert\log\left(1-e\left(\frac{z}{\omega_1}\right)\right)\bigg\vert\\
&-\frac{1}{2\pi}\log\bigg\vert\prod\limits_{n=1}^{\infty}\left(1-e\left(\frac{n\omega_2+z}{\omega_1}\right)\right)
\left(1-e\left(\frac{n\omega_2-z}{\omega_1}\right)\right)\bigg\vert,
\end{aligned}
\end{equation}
where $z=x-y$ and $x,y,z$ are numbers in the complex plane; $|T|=\mbox{Im}\ \bar{\omega}_1\omega_2=1$. It is easy to verify that $G(z)=G(-z)$. In particular, Chen, Lin and Wang in \cite{clw} showed that $G(z)$ is evenly symmetric about both axes if $T$ is a rectangular torus, i.e. $G(z)=G(-z)=G(\bar{z})$. We also define the regular part of the Green's function:
\begin{equation}
\label{2.6}
R(x,y)=R(z):=G(x,y)+\frac{1}{2\pi}\log{(d(x,y))},
\end{equation}
where $d(x,y)$ is defined in the Remark \ref{notation}.

\medskip

Let $U_{n,i}$ be the standard bubble at $x_{n,i}$, i.e.
\begin{equation}
\label{2.7}
U_{n,i}(x)=\log{\left(\frac{e^{\lambda_{n,i}}}{(1+\frac{\rho_n}{8}e^{\lambda_{n,i}}(d(x,x_{n,i}))^2)^2}\right)}
,\quad i=1,2.
\end{equation}
Chen and Lin in \cite{cl} obtained some sharp estimates on the error term $\eta_{n,i}$, which is defined as follows
\begin{equation}
\label{2.8}
\eta_{n,i}(x)=u_n(x)-U_{n,i}(x)-(G^{\ast}_i(x)-G^{\ast}_i(x_{n,i})),\quad x\in B(x_{n,i},\delta).
\end{equation}
For $x\in B(x_{n,i},\delta)$, they proved
\begin{equation}
\label{2.9}
\begin{aligned}
\eta_{n,i}(x)=~&-\frac{128\pi}{\rho_n}e^{-\lambda_{n,i}}\left[\log{\left(R_{n,i}|x-x_{n,i}|+2\right)}\right]^2 \\&+O\left(\log{\left(R_{n,i}|x-x_{n,i}|+2\right)}e^{-\lambda_{n,i}}\right)\\
&+O\left(\lambda_{n,i}e^{-\lambda_{n,i}}\right)=O(\lambda^2_{n,i}e^{-\lambda_{n,i}}),\quad i=1,2,
\end{aligned}
\end{equation}
where $R_{n,i}=\sqrt{\rho_n e^{\lambda_{n,i}}/8}$. It has also been proved in \cite{l} that there are constants $c>0$ and $c_{\delta}>0$ such that,
\begin{equation}
\label{2.10}
|\lambda_n-\lambda_{n,i}|\leq c~\textrm{ for }~i=1,2,\quad
|u_n(x)+\lambda_n|\leq c_{\delta}~\textrm{ for }~x\in T\setminus \bigcup\limits_{i=1}^2 B(p_i,\delta).
\end{equation}
More precisely, see \cite[Section 3]{cl}, we have
\begin{equation}
\label{2.11}
e^{\lambda_{n,i}}e^{G^{\ast}_i(x_{n,i})}=
e^{\lambda_{n,1}}e^{G^{\ast}_1(x_{n,1})}\left(1+O\left(e^{-\frac{\lambda_{n,1}}{2}}\right)\right),\quad i=1,2.
\end{equation}
In particular, see \cite[Theorem 1.4]{cl}, the following estimate holds,
\begin{equation}
\label{2.12}
\begin{aligned}
&\lambda_{n,i}+\int_{T} u_n(x) dx + 2\log{\left(\frac{\rho_n}{8}\right)} + G^{\ast}_i(x_{n,i})\\
&=-\frac{32\pi}{\rho_n}\lambda^2_{n,i} e^{-\lambda_{n,i}} + O\left(\lambda_{n,i}e^{-\lambda_{n,i}}\right),\quad i=1,2.
\end{aligned}
\end{equation}

Notice that in \cite[Lemma 5.5]{cl}, the Pohozaev identity is used to derive that
\begin{equation}
\label{2.13}
\nabla G^{\ast}_i\big\vert_{x=x_{n,i}} = O\left(\lambda_{n,i}e^{-\lambda_{n,i}}\right),\quad i=1,2.
\end{equation}
Together with the non-degeneracy assumption on the critical point $\vec{p}$, Bartolucci et al. \cite{bjly} concluded that
\begin{equation}
\label{2.14}
|x_{n,i}-p_i|=O(\lambda_{n,i}e^{-\lambda_{n,i}}),\quad i=1,2.
\end{equation}
However, in our case, the critical point $\vec{p}=(p_1,p_2)$ where $p_2-p_1$ is one of three half-periods (namely $\omega_1/2$, $\omega_2/2$ and $\frac{\omega_1+\omega_2}{2}$) is a degenerate critical point of $G(x,y)$. Thus, we need a different way to get the above estimate on the distance between the local maximum point and the blow-up point. By imposing the symmetry condition on $u_n$, we are able to show that (\ref{2.14}) holds
in our setting as well.

\begin{proposition}
\label{th2.1}
Suppose that $\{u_n\}$ is a sequence of blow-up solutions of (\ref{2.1}), satisfying $\rho_n\rightarrow 16\pi^+$ and $u_n(z)=u_n(-z)$. Then $\{u_n\}$ blows-up at two points $p_1=0$ and $p_2$ which is any half-period. Furthermore, if $x_{n,i},~i=1,2$ is the local maximum point as defined in (\ref{2.3}), then we have
\begin{equation*}
d(x_{n,i},p_i) = O(\lambda_{n,i}e^{-\lambda_{n,i}}).
\end{equation*}
\end{proposition}
\begin{proof}
 Ma and Wei in \cite{mw} proved that the blow-up points $(p_1,\cdots,p_m)$ of solutions to the corresponding Dirichlet problem of mean field type must be a critical point of the $m$-vortex Hamiltonian $f_m$, and they also pointed out that the same conclusion would also hold for (\ref{1.2}). Chen and Lin in \cite[Estimate B]{cl} obtained a similar conclusion in the manifold setting by using the Pohozaev identity. Therefore, the only possible two-point blow-up would happen at the critical points of $G(x,y)$. By assumption of the symmetry of solutions, the blow-up points must be one of the three cases stated in Theorem \ref{th2.1}.

Then, it suffices to prove the estimate on $d(x_{n,i},p_i),~i=1,2$. Without loss of generality, let us consider the case $i=1$. By (\ref{2.8}), we can write $u(x)=U_{n,1}(x)+G^{\ast}_1(x)-G^{\ast}_1(x_{n,1})+\eta_{n,1}(x)$ for $x\in B(x_{n,1},\delta)$. Since $x_{n,1}\rightarrow 0$, we are always able to choose $n$ sufficiently large such that $-x_{n,1}\in B(x_{n,1},\delta)$. Thus, using the fact that $u(x_{n,1})=u(-x_{n,1})$ together with (\ref{2.9}) and (\ref{2.13}), we conclude that $|x_{n,1}|=O\left(\lambda_{n,1}e^{-\lambda_{n,1}}\right)$. Note that $G^{\ast}_1$ is a smooth function.
\end{proof}

\begin{remark}\label{cr.pts}
Recently, Chen, Kuo, Lin and Wang in \cite{cklw} showed that $G(z)$ might have an extra pair of ``non-trivial" critical points other than the three half periods points for a class of flat tori. Moreover, the ``non-trivial" critical points are always non-degenerate. Based on these observations, one should be able to construct two distinct families of blow-up solutions which are not evenly symmetric. It is also possible to prove ``local uniqueness" for solutions that blow-up at the origin and one of the ``non-trivial" critical points.
\end{remark}

Let us define the local masses corresponding to the blow-up of $u_n$ at $p_i$, $i=1,2$:
\begin{equation}
\label{2.15}
\rho_{n,i}=\rho_n \int_{B(p_i,\delta)} e^{u_n}dx,\quad i=1,2.
\end{equation}
We have the following estimate on $\rho_{n,i},~i=1,2$, see \cite[Section 3]{cl}
\begin{equation}
\label{2.16}
\rho_{n,i}-8\pi=16\pi \lambda_{n,i} e^{-\lambda_{n,i}} + O\left(e^{-\lambda_{n,i}}\right),\quad i=1,2.
\end{equation}
For the total mass, see \cite[Theorem 1.1]{cl}, we have
\begin{equation}
\label{2.17}
\begin{aligned}
\rho_n-16\pi=16\pi\sum\limits_{i=1}^2 \lambda_{n,i} e^{-\lambda_{n,i}} + O(e^{-\lambda_{n,i}})
=\frac{\lambda_{n,1}e^{-\lambda_{n,1}}}{e^{G^{\ast}_1(p_1)}}l(\vec{p})+ O\left(e^{-\lambda_{n,1}}\right),
\end{aligned}
\end{equation}
where $l(\vec{p})$ is the quantity defined in (\ref{1.9}). In particular, we recall $l(\vec{p})=32\pi e^{G^{\ast}_1(p_1)}\neq 0$. We would like to remark that a more refined estimate involving $D(\vec{q})$ on the total mass has been derived in \cite[Theorem 1.3]{bjly} which is crucial in the case where $l(\vec{p})$ vanishes.

\medskip

We will also need the asymptotic behaviour of $u_n$ outside the union of the balls $B(p_i,\delta)$, $i=1,2$. In particular, we consider the ``outer" error defined as follows
\begin{equation}
\label{2.18}
\omega_n(x)=u_n(x)-\sum\limits_{i=1}^2 \rho_{n,i} G(x,x_{n,i})-\int_{T}u_n dx.
\end{equation}
It is already proved in \cite[Estimate A]{cl} that
\begin{equation}
\label{2.19}
\omega_n= O(e^{-\lambda_n/2})~\textrm{ in }~C^1\Big(T\setminus \bigcup\limits_{i=1}^2 B(p_i,\delta)\Big).
\end{equation}

\medskip
\section{Uniqueness of Blow-up Solutions}\label{sec3}
In this section we will prove Theorem \ref{th2} by contradiction. Suppose that (\ref{2.1}) has two distinct solutions $u^{(1)}_n$ and $u^{(2)}_n$ which blow-up at $p_i$, $i=1,2$. Let us use $x^{(\ell)}_{n,i}$, $\lambda^{(\ell)}_n$, $\lambda^{(\ell)}_{n,i}$, $U^{(\ell)}_{n,i}$,  $\eta^{(\ell)}_{n,i}$, $R^{(\ell)}_{n,i}$, $\rho^{(\ell)}_{n,i}$ and $\omega^{(\ell)}_n$ to denote $x_{n,i}$, $\lambda_n$, $\lambda_{n,i}$, $U_{n,i}$, $\eta_{n,i}$, $R_{n,i}$, $\rho_{n,i}$, $\omega_n$, as defined in Section \ref{sec2}, corresponding to $u^{(\ell)}_n$, $\ell=1,2$, respectively.

\medskip

As in \cite{bjly,ly} we consider the normalized difference of the two solutions
\begin{equation}
\label{3.1}
\xi_n(x)=\frac{u^{(1)}_n(x)-u^{(2)}_n(x)}{\norm{u^{(1)}_n-u^{(2)}_n}_{L^{\infty}(T)}}.
\end{equation}
Roughly speaking, our aim is to show that the projections of $\xi_n$, $n\rightarrow+\infty$, on the three kernel functions (introduced in \eqref{1.6}) of the linearized operator \eqref{lin} are zero and then derive a contradiction by showing that $\xi_n\not\equiv0$, $n\rightarrow+\infty$. The plan is the following:
\begin{itemize}
  \item[(1)] study the asymptotic behavior of $\xi_n$ inside and outside the blow-up disks,
  \item[(2)] use a suitable Pohozaev identity to show the projection of $\xi_n$ on the radial part kernel vanishes,
	\item[(3)] exploit the evenly symmetric property to show the projections of $\xi_n$ on the kernels related to translations are zero and finally prove Theorem \ref{th2}.
\end{itemize}

\medskip

\subsection{Some Useful Estimates.}
We start by studying the asymptotic behavior of $\xi_n$. This part follows closely \cite{bjly} jointly with Proposition \ref{th2.1}, so we skip the computations and refer the reader to \cite{bjly} for full details.

\begin{lemma}
\label{le3.1} \emph{(\cite{bjly})}
There exists a constant $C>0$ such that
\begin{equation}
\label{3.2}
|\lambda^{(1)}_{n,i}-\lambda^{(2)}_{n,i}|\leq C\left(\frac{1}{\lambda^{(1)}_{n,1}}+\frac{1}{\lambda^{(2)}_{n,1}}\right),~i=1,2.
\end{equation}
Moreover,
\begin{equation}
\label{3.3}
\norm{u^{(1)}_n-u^{(2)}_n}_{L^{\infty}(T)} = O\left(|\lambda^{(1)}_{n,1}-\lambda^{(2)}_{n,1}|+\sum\limits_{\ell=1}^2 \lambda^{(\ell)}_{n,1}e^{-\frac{\lambda^{(\ell)}_{n,1}}{2}}\right).
\end{equation}
\end{lemma}

\medskip

It is easy to see that $\xi_n$ satisfies
\begin{equation}
\label{3.4}
\Delta \xi_n + f^{\ast}_n(x)=\Delta \xi_n + \rho_n c_n(x)\xi_n(x)=0,
\end{equation}
where
\begin{equation}
\label{3.5}
f^{\ast}_n(x)=\frac{\rho_n }{\norm{u^{(1)}_n-u^{(2)}_n}_{L^{\infty}(T)}}\left(e^{u^{(1)}_n(x)}-e^{u^{(2)}_n(x)}\right),
\end{equation}
and
\begin{equation}
\label{3.6}
c_n(x)=\frac{e^{u^{(1)}_n(x)}-e^{u^{(2)}_n(x)}}{u^{(1)}_n(x)-u^{(2)}_n(x)}
=e^{u^{(1)}_n}\left(1+O(\norm{u^{(1)}_n-u^{(2)}_n}_{L^{\infty}(T)})\right).
\end{equation}

Next, in the following lemma we give the description of $\xi_n$ both inside the bubbling disc $B(p_i,\delta),~i=1,2$, and the asymptotic behavior $\xi_n$ away from the blow-up points $p_i,~i=1,2.$

\begin{lemma}
\label{le3.2} \emph{(\cite{bjly})}
Let
$$\xi_{n,i}(z)=\xi_{n}\left(e^{-\frac{\lambda^{(1)}_{n,i}}{2}}z+x^{(1)}_{n,i}\right),~|z|<\delta e^{-\frac{\lambda_{n,i}^{(1)}}{2}},~i=1,2,$$
then there exist constants $b_{i,0},b_{i,1},b_{i,2}$ such that
\begin{equation*}
\xi_{n,i}(z)\rightarrow b_{i,0} \psi_{i,0}(z) + b_{i,1} \psi_{i,1}(z) + b_{i,2} \psi_{i,2}(z)
\end{equation*}
in $C_{\textrm{loc}}(\mathbb{R}^2)$, where
\begin{equation*}
\psi_{i,0}(z)=\frac{1-2\pi |z|^2}{1+2\pi |z|^2}, \textrm{ }\psi_{i,1}(z)=\frac{\sqrt{2\pi}z_1}{1+2\pi |z|^2}, \textrm{ }\psi_{i,2}(z)=\frac{\sqrt{2\pi} z_2}{1+2\pi |z|^2}.
\end{equation*}
Furthermore $b_{1,0}=b_{2,0}=b_0$ for some constant $b_0$, and
\begin{equation*}
\xi_n(x)=-b_0+o(1),\quad\forall x\in T\setminus\bigcup_{i=1}^2B(p_i,e^{-\frac{\lambda_{n,i}^{(1)}}{2}}R),
\end{equation*}
for some $R>0$ sufficiently large.
\end{lemma}

\medskip

For the proof of Lemma \ref{le3.1} and Lemma \ref{le3.2}, we refer the readers to \cite[Lemma 3.1-Lemma 3.4]{bjly}.

\medskip

\subsection{Radial part kernel.}
We prove here that the projection of $\xi_n$ on the radial part kernel vanishes. Since by Lemma \ref{le3.2} we have $b_{1,0}=b_{2,0}=b_0$, we need to show that $b_0=0$. Then, for $i=1,2,$ let
\begin{equation}
\label{3.7}
\phi_{n,i}(y)=\frac{\rho_n}{2}\left(R(x^{(1)}_{n,i},y)-R(x^{(1)}_{n,i},x^{(1)}_{n,i})
+G(x^{(1)}_{n,l},y)-G(x^{(1)}_{n,l},x^{(1)}_{n,i})\right),
\end{equation}
where $l\neq i$, and
\begin{equation}
\label{3.8}
v^{(\ell)}_{n,i}(y)=u^{(\ell)}_n(y)-\phi_{n,i}(y),\quad \ell=1,2.
\end{equation}
To show $b_0=0$ we need the following Pohozaev identity from \cite[Lemma 3.6]{bjly}.
\begin{lemma}
\label{le3.3} \emph{(\cite{bjly})}
For any fixed $r\in(0,\delta)$, we have
\begin{equation}
\label{3.9}
\begin{aligned}
&\frac{1}{2}\int_{\partial B(x^{(1)}_{n,i},r)} r\langle \nabla (v^{(1)}_{n,i}+v^{(2)}_{n,i}), \nabla \xi_n \rangle - \int_{\partial B(x^{(1)}_{n,i},r)} r \langle \nu, \nabla (v^{(1)}_{n,i}+v^{(2)}_{n,i})\rangle \langle \nu, \nabla\xi_n\rangle\\
&=\int_{\partial B(x^{(1)}_{n,i},r)}
\frac{r\rho_n}{\norm{v^{(1)}_{n,i}-v^{(2)}_{n,i}}_{L^{\infty}(T)}}
\left(e^{v^{(1)}_{n,i}+\phi_{n,i}}-e^{v^{(2)}_{n,i}+\phi_{n,i}}\right)\\
&\quad -\int_{B(x^{(1)}_{n,i},r)}\frac{\rho_n \left(e^{v^{(1)}_{n,i}+\phi_{n,i}}-e^{v^{(2)}_{n,i}+\phi_{n,i}}\right)}
{\norm{v^{(1)}_{n,i}-v^{(2)}_{n,i}}_{L^{\infty}(\tilde{T})}}\left(2+\langle \nabla\phi_{n,i}, x-x^{(1)}_{n,i}\rangle\right).
\end{aligned}
\end{equation}
\end{lemma}

\medskip

Next, we can follow the computations in \cite[Lemma 4.2-Lemma 4.3]{bjly} jointly with Proposition \ref{th2.1} to get the estimate on both sides of \eqref{3.9}. For the left hand side of \eqref{3.9}, we have
\begin{equation}
\label{3.10}
\begin{aligned}
\textrm{LHS of }(\ref{3.9}) =& -4A_{n,i}-\frac{256 b_0 e^{-\lambda^{(1)}_{n,1}+G^{\ast}_i(p_i)}}{\rho_n e^{G^{\ast}_1(p_1)}}\int_{M_i\setminus B(p_i,r)} e^{\Phi_i(x,\vec{p})}dx\\
&+o(e^{-\frac{\lambda_{n,i}^{(1)}}{2}}\sum\limits_{l=1}^2 |A_{n,l}|)+o(e^{-\lambda_{n,i}^{(1)}}),\quad i=1,2,
\end{aligned}
\end{equation}
for fixed $r\in(0,r_0)$, where
\begin{equation*}
A_{n,i}=\int_{M_i}\frac{\rho_n}{\norm{u_n^{(1)}-u_n^{(2)}}_{L^\infty(T)}}\left(e^{u_n^{(1)}}-e^{u_n^{(2)}}\right),
\end{equation*}
$\Phi_i$ is defined in \eqref{1.12} and $r_0$ is introduced after \eqref{1.9}. For the right hand side of \eqref{3.9}, we have
\begin{equation}
\label{3.11}
\begin{aligned}
\textrm{RHS of }(\ref{3.9})=
&-e^{-\lambda_{n,1}^{(1)}}\left(\frac{128b_0 e^{G^{\ast}_i(p_i)}}{\rho_n e^{G^{\ast}_1}(p_1)}\frac{\pi}{r^2}
+\frac{512\pi^2b_0 e^{G^{\ast}_i(p_i)}}{\rho_n e^{G^{\ast}_1(p_1)}}\right)\\
&-e^{-\lambda^{(1)}_{n,1}} \frac{128 b_0}{\rho_n e^{G^{\ast}_1(p_1)}}\int_{M_i\setminus B(p_i,r_i)}e^{G^{\ast}_i(p_i)+\Phi_i(x,\vec{p})}dx\\
&-e^{-\lambda^{(1)}_{n,1}}\left(\lambda^{(1)}_{n,1}+\log{\left(\frac{\rho_n e^{G^{\ast}_1(p_1)}}{8 e^{G^{\ast}_i(p_i)}}r^2\right)}-2\right)\Pi e^{G^{\ast}_i(p_i)}\\
&+O(e^{-\lambda_{n,1}^{(1)}})(r+R^{-1})+o(e^{-\lambda_{n,i}^{(1)}})(\log r+\log R)\\
&+O(\sum_l|A_{n,l}|(R^{-1}e^{-\frac{\lambda_{n,i}^{(1)}}{2}}+e^{-\lambda_{n,i}^{(1)}}
(\lambda_{n,i}^{(1)}+\log r)))\\
&+o(e^{-2\lambda_{n,i}^{(1)}}r^{-2}),
\end{aligned}
\end{equation}
for any $r\in(0,1)$ and $R$ sufficiently large. Here
\begin{equation*}
\Pi=\frac{512\pi^2 \left((\int_{T}\xi_n)- \frac{\norm{u^{(1)}_n-u^{(2)}_n}_{L^{\infty}(T)}}{2}
(\int_{T}\xi_n)^2\right)}{\rho_n e^{G^{\ast}_1(p_1)}}.
\end{equation*}
With \eqref{3.10} and \eqref{3.11}, we are now able to show $b_0=0$.

\begin{lemma}
\label{le3.4}
It holds $b_0=b_{1,0}=b_{2,0}=0$.
\end{lemma}

\begin{proof}
By \eqref{3.9}-\eqref{3.11}, together with \eqref{2.10} and \eqref{2.12}, we have
\begin{equation*}
\begin{aligned}
&-4 A_{n,i} + O(e^{-\frac{\lambda_{n,1}^{(1)}}{2}}\sum\limits_{l=1}^2 |A_{n,l}|)+o(e^{-\lambda_{n,1}^{(1)}})\\
&=-2 A_{n,i} + O(\lambda^{(1)}_{n,i}e^{-\lambda^{(1)}_{n,i}})+ O(r^{-2}e^{-\lambda_{n,i}^{(1)}})+o(e^{-\lambda_{n,i}^{(1)}}\log{R}),
\end{aligned}
\end{equation*}
which implies that
\begin{equation}
\label{3.12}
A_{n,i}=o(e^{-\frac{\lambda_{n,1}^{(1)}}{2}}),\quad i=1,2.
\end{equation}
For any $r>0$, let $r_i=r\sqrt{8e^{G_i^*(p_i)}},~i=1,2$. For each point $p_i$, we choose $r=r_i$ in \eqref{3.9}. By \eqref{3.9}-\eqref{3.12}, we have
\begin{equation}
\label{3.13}
\begin{aligned}
&\sum_{i=1}^2\left[-4A_{n,i}-\frac{256 b_0 e^{-\lambda^{(1)}_{n,1}+G^{\ast}_i(p_i)}}{\rho_n e^{G^{\ast}_1(p_1)}}\int_{M_i\setminus B(p_i,r_i)} e^{\Phi_i(x,\vec{p})}dx\right]\\
&=\sum_{i=1}^2\left[-e^{-\lambda_{n,1}^{(1)}}\left(\frac{128b_0 e^{G^{\ast}_i(p_i)}}{\rho_n e^{G^{\ast}_1(p_1)}}\frac{\pi}{r^2_i}
+\frac{512\pi^2b_0 e^{G^{\ast}_i(p_i)}}{\rho_n e^{G^{\ast}_1(p_1)}}\right)\right.\\
&\quad-e^{-\lambda^{(1)}_{n,1}} \frac{128 b_0}{\rho_n e^{G^{\ast}_1(p_1)}}\int_{M_i\setminus B(p_i,r_i)}e^{G^{\ast}_i(p_i)+\Phi_i(x,\vec{p})}dx\\
&\quad\left.-e^{-\lambda^{(1)}_{n,1}}\left(\lambda^{(1)}_{n,1}+\log{\left(\frac{\rho_n e^{G^{\ast}_1(p_1)}}{8 e^{G^{\ast}_i(p_i)}}r^2_i\right)}-2\right)\Pi e^{G^{\ast}_i(p_i)}\right]\\
&\quad+o(e^{-\lambda_{n,1}^{(1)}}(r^{-2}+\log R+\log r))+O(e^{-\lambda_{n,1}^{(1)}}(r+R^{-1})),
\end{aligned}
\end{equation}
for any $r\in(0,1),~R>1$ sufficiently large. Then using the fact that $A_{n,1}+A_{n,2}=0$, we deduce that
\begin{equation}
\label{3.14}
\begin{aligned}
&-\frac{256 b_0 e^{-\lambda^{(1)}_{n,1}}}{\rho_n e^{G^{\ast}_1(p_1)}}\sum_{i=1}^2e^{G^{\ast}_i(p_i)}\int_{M_i\setminus B(p_i,r_i)} e^{\Phi_i(x,\vec{p})}dx\\
&=-e^{-\lambda^{(1)}_{n,1}}\frac{128 b_0}{\rho_n e^{G^{\ast}_1(p_1)}} \sum\limits_{i=1}^2 e^{G^{\ast}_i(p_i)}\frac{\pi}{r^2_i}-e^{-\lambda^{(1)}_{n,1}}\frac{32\pi b_0 l(\vec{p})}{\rho_n e^{G^{\ast}_1(p_1)}}\\
&\quad -e^{-\lambda^{(1)}_{n,1}}\frac{128 b_0}{\rho_n e^{G^{\ast}_1(p_1)}}\sum\limits_{i=1}^2  e^{G^{\ast}_i(p_i)}\int_{M_i\setminus B(p_i,r_i)}e^{\Phi_i(y,\vec{p})}dy\\
&\quad -e^{-\lambda^{(1)}_{n,1}}\frac{\Pi}{16\pi}l(\vec{p})\left(\lambda^{(1)}_{n,1}
+\log({\rho_ne^{G^{\ast}_1(p_1)}r^2})-2\right)\\
&\quad+o(e^{-\lambda_{n,1}^{(1)}}(r^{-2}+\log R+\log r))+O(e^{-\lambda_{n,1}^{(1)}}(r+R^{-1})).
\end{aligned}
\end{equation}
By Lemma \ref{le3.2}, we have $\int_{T}\xi_n=-b_0+o(1)$. We divide \eqref{3.14} by $\lambda^{(1)}_{n,1}e^{-\lambda^{(1)}_{n,1}}$ and derive that $l(\vec{p})b_0=o(1)$. Therefore, we conclude that $b_0=0$ since $l(\vec{p})\neq 0$.
\end{proof}

\medskip

\subsection{The conclusion.}
We are now in the position to prove Theorem \ref{th2}.

\begin{proof}[Proof of Theorem \ref{th2}]
We already know that the projections on the radial part kernel are zero, i.e. $b_{i,0}=0$ for $i=1,2$, see Lemma \ref{le3.4}. Let us show now that the projections $b_{i,k}$, $i,k=1,2$, on the kernel related to translations are zero. Using the fact that both $u^{(1)}_n$ and $u^{(2)}_n$ are evenly symmetric, we can see that $u^{(1)}_n(x+p_2)$ and $u^{(2)}_n(x+p_2)$ are also evenly symmetric. As a consequence, the projection of the normalized difference on the kernel related to translations vanishes automatically, i.e.
\begin{equation}
\label{3.15}
b_{i,k}=0,~i,k=1,2.
\end{equation}

\medskip

We will conclude now by showing that $\xi_n\not\equiv0$, $n\rightarrow+\infty$. Let $x^{\ast}_n$ be a maximum point of $\xi_n$,  then we have
\begin{equation}
\label{3.16}
|\xi_n(x^{\ast}_n)|=1.
\end{equation}
Therefore, by Lemma \ref{le3.2} and Lemma \ref{le3.4}, we find that $\lim_{n\rightarrow \infty}x_n^*=p_i$ for some $i$. Moreover, in view of Lemma \ref{le3.4} and the fact that $b_{i,k}=0$ for $k=1,2$, we deduce that
\begin{equation}
\label{3.17}
\lim\limits_{n\rightarrow \infty}e^{\frac{\lambda_{n,i}^{(1)}}{2}} s_n=\infty,
\end{equation}
where $s_n=|x^{\ast}_n-x^{(1)}_{n,i}|$. We set $\tilde{\xi}_n(x)=\xi_n(s_n x + x^{(1)}_{n,i})$, then (\ref{2.9}) and (\ref{3.4}) imply that $\tilde{\xi}_n$ satisfies
\begin{equation*}
\Delta \tilde{\xi}_n + \rho_n s^2_n c_n(s_n x + x^{(1)}_{n,i})\tilde{\xi}_n
=\Delta \tilde{\xi}_n+ \frac{\rho_n s^2_n e^{\lambda^{(1)}_{n,1}}(1+O(s_n|x|)+o(1))}
{(1+\frac{\rho_n}{8}e^{\lambda^{(1)}_{n,1}}s^2_n|x|^2)^2}\tilde{\xi}_n.
\end{equation*}
On the other hand, we have
\begin{equation}
\label{3.18}
\bigg|\tilde{\xi}_n\left(\frac{x^{\ast}_n-x^{(1)}_{n,i}}{s_n}\right)\bigg|=|\xi_n(x^{\ast}_n)|=1.
\end{equation}
In view of (\ref{3.17}) and $|\tilde{\xi}_n|\leq 1$, we see that if $\tilde{\xi}_n\rightarrow \tilde{\xi}_0$ on any compact subset of $\mathbb{R}^2\setminus \{0\}$, then $\Delta\tilde{\xi}_0=0$ in $\mathbb{R}^2 \setminus \{0\}$. Since $\tilde{\xi}_0$ is also bounded, then we can conclude that $\tilde{\xi}_0$ is smooth and harmonic on entire $\mathbb{R}^2$. Hence $\tilde{\xi}_0$ is a constant. Since $\frac{|x^{\ast}_n-x^{(1)}_{n,i}|}{s_n}=1$ and in view of (\ref{3.18}), we obtain that either $\tilde{\xi}_0=1$ or $\tilde{\xi}_0=-1$. In particular, we have that $|\tilde{\xi}_n|\geq \frac{1}{2}$ for $s_n\leq |x-x^{(1)}_{n,i}|\leq 2 s_n$, which contradicts to the second conclusion of Lemma \ref{le3.2} because of the facts that $\lim_{n\rightarrow \infty} e^{\frac{\lambda_{n,i}^{(1)}}{2}} s_n=\infty$, $\lim_{n\rightarrow \infty}s_n=0$ and $b_0=b_{j,0}=0$. This finishes the proof of Theorem \ref{th1}.
\end{proof}

\medskip
\section{Existence of Blow-up Solutions}\label{sec4}
In this section we will use a Lyapunov-type reduction method to construct blow-up solutions to (\ref{1.1}). Since the proof of Theorem \ref{th1} follows along the same line as the arguments used in \cite[Theorem 2.1]{gh} and \cite[Theorem 2.3]{cgh}, we shall give the key steps and refer the readers to the above two papers for details.

\medskip

\subsection{Approximate Solution}
\noindent We start with an approximate solution of the equation (\ref{1.1}) and obtain some estimates of this approximate solution. Let $R_0>0$ be a small fixed number and $\eta$ be a cut-off function such that
\begin{equation*}
\eta(s)=\begin{cases}
1~&\mathrm{for}~ s\leq1,\\
0~&\mathrm{for}~ s\geq2,
\end{cases}
\quad 0\leq \eta(x)\leq 1,\quad \abs{\eta'(s)}\leq 2.
\end{equation*}
Let
\begin{equation}
\label{4.1}
\eta_{t,a}(x)=\eta\left(\frac{d(x,a)}{t}\right),\quad \forall a\in T~\mathrm{and}~t>0.
\end{equation}
Given $\epsilon\in(0,\epsilon_0)$, for later purposes we choose $\lambda>0$ such that
\begin{equation}
\label{4.2}
16\pi\lambda e^{-\lambda}<\epsilon<64\pi\lambda e^{-\lambda},
\end{equation}
or equivalently,
\begin{equation}
\label{4.3}
\lambda_1(\epsilon)<\lambda<\lambda_2(\epsilon),
\end{equation}
where
$$16\pi\lambda_1(\epsilon)e^{-\lambda_1(\epsilon)}=\epsilon\quad\mathrm{and}\quad
64\pi\lambda_2(\epsilon)e^{-\lambda_2(\epsilon)}=\epsilon.$$
Since our ansatz will resemble a bubble function around each blow-up point we start by letting $w_{\lambda,i}$ be the solution of the following localized equation:
\begin{equation}
\label{4.4}
-\Delta w_{\lambda,i}=\frac{16\pi e^{\lambda}}{(1+2\pi e^{\lambda}(d(x,p_i))^2)^2}\eta_{R_0,p_i}-m_0,
\quad \int_{T} w_{\lambda,i}=0,
\end{equation}
where
\begin{equation}
\label{4.5}
m_0=\int_{T} \frac{16\pi e^{\lambda}}{(1+2\pi e^{\lambda}(d(x,p_i))^2)^2}\eta_{R_0,p_i}dx
=8\pi+O(e^{-\lambda}),\quad i=1,2,
\end{equation}
where we used the fact that $|T|=1.$

\medskip

In order to have a good approximation we need the following estimates. Let us calculate the value of $w_{\lambda,i}(p_i),~i=1,2$:
\begin{equation}
\label{4.6}
\begin{aligned}
w_{\lambda,i}(p_i)=&\int_{T} G(p_i,y)\left[\frac{16\pi e^{\lambda}}{(1+2\pi e^{\lambda}|y|^2)^2}\eta_{R_0,p_i}-m_0\right]dy\\
=&\int_{B(0,R_0)}\left[-\frac{1}{2\pi}\log{|y|}+R(0,y)\right]\frac{16\pi e^{\lambda}}{(1 + 2\pi e^{\lambda}|y|^2)^2} dy + O(e^{-\lambda})\\
=&\int_{B(0,~e^{\frac{\lambda}{2}}R_0)}\left[\frac{\lambda}{4\pi}-\frac{1}{2\pi}\log{|z|} + R(0,e^{-\frac{\lambda}{2}}z)\right]\frac{16\pi dz}{(1+2\pi |z|^2)^2} + O(e^{-\lambda})\\
=&\int_{B(0,~e^{\frac{\lambda}{2}}R_0)}\left[\frac{\lambda}{4\pi}-\frac{1}{2\pi}\log{|z|} + R(0,0)\right]\frac{16\pi dz}{(1+2\pi |z|^2)^2}\\
&+e^{-\lambda}\int_{B(0,~e^{\frac{\lambda}{2}}R_0)} \frac{4\pi|z|^2}{(1+2\pi|z|^2)^2}dz + O(e^{-\lambda})\\
=&~2\lambda + 2\log{(2\pi)}+8\pi R(0,0) +\lambda e^{-\lambda} + O(e^{-\lambda}).
\end{aligned}
\end{equation}
We can also estimate the value of $w_{\lambda,i}$ near $p_i$:
\begin{equation}
\label{4.7}
\begin{aligned}
&w_{\lambda,i}(p_i+e^{-\frac{\lambda}{2}}z)-w_{\lambda,i}(p_i)\\
&=\int_{T}\left[G(p_i+e^{-\frac{\lambda}{2}} z,y)-G(p_i,y)\right]
\frac{16\pi e^{\lambda}}{(1+2\pi e^{\lambda}|y-p_i|^2)^2}\eta_{R_0,p_i}dy\\
&=\int_{B(0, ~e^{\frac{\lambda}{2}}R_0)}-\frac{1}{2\pi}\left[\log{|z-z'|-\log{|z'|}}\right]\frac{16\pi dz'}{(1+2\pi |z'|^2)^2}\\
&\quad +\int_{B(0, ~e^{\frac{\lambda}{2}}R_0)}
\left[R\left(e^{\frac{\lambda}{2}}(z'-z)\right)-R\left(e^{\frac{\lambda}{2}}z'\right)\right]
\frac{16\pi dz'}{(1+2\pi |z'|^2)^2} + O\left(e^{-\frac{3}{2}\lambda}|z|\right)\\
&=\log{\left(\frac{1}{(1+2\pi |z|^2)^2}\right)}+ e^{-\lambda} Q(z_1,z_2) + O\left(e^{-\frac{3}{2}\lambda}|z|^3\right) + O\left(\lambda e^{-\frac{3}{2}\lambda}|z|\right)
\end{aligned}
\end{equation}
for $|z|<e^{\frac{\lambda}{2}}R_0$, where $Q(z_1,z_2)$ is a quadratic form depending only on $\nabla^2 R(z)\vert_{z=0}$ with the property $\Delta Q=8\pi.$ Note here we use the fact that $\Delta R=1$. For $|z|\geq 2 e^{\frac{\lambda}{2}}R_0$, i.e., $d(x,p_i)\geq 2 R_0$, we have
\begin{equation}
\label{4.8}
\begin{aligned}
&w_{\lambda,i}(x)=w_{\lambda,i}(p_i + e^{-\frac{\lambda}{2}}z)=\int_{T} G(x,x')\frac{16\pi e^{\lambda}}{(1+2\pi e^{\lambda}|x'-p_i|^2)^2}\eta_{R_0,p_i}dx'\\
&=\int_{B(0,~e^{\frac{\lambda}{2}}R_0)} G(x,p_i+e^{-\frac{\lambda}{2}} z')\frac{16\pi dz'}{(1+2\pi |z'|^2)^2} +O(e^{-\lambda})\\
&=\int_{B(0,~e^{\frac{\lambda}{2}}R_0)}\left[G(p_i-x)+e^{-\frac{\lambda}{2}}\nabla G(p_i-x)\cdot z' + e^{-\lambda}\frac{|z'|^2}{4}\right]\frac{16\pi dz'}{(1+2\pi |z'|^2)^2} + O(e^{-\lambda})\\
&=8\pi G(p_i-x) + \pi e^{-\lambda}\int_0^{e^{\frac{\lambda}{2}}R_0}\frac{8\pi r^3}{(1+2\pi r^2)^2}dr + O(e^{-\lambda})\\
&=8\pi G(p_i-x) + \lambda e^{-\lambda} + O(e^{-\lambda}),
\end{aligned}
\end{equation}
\medskip

Letting $\overline{w}_{\lambda}=-\lambda+\log{(\frac{4}{\pi})}-8\pi R(0) -8\pi G(p_1,p_2)$ we are now ready to provide an ansatz for the solution of (\ref{1.1})
\begin{equation}
\label{4.9}
w_{\lambda}=\sum\limits_{i=1}^2 w_{\lambda,i} + \overline{w}_{\lambda}.
\end{equation}
Combining (\ref{4.6})-(\ref{4.8}), we obtain the following lemma concerning the asymptotic behavior of $w_{\lambda}$ near the blow-up points $p_1$ and $p_2$:
\begin{lemma}
\label{le4.1}
For $z\in B(0,e^{\frac{\lambda}{2}}R_0)$, we have
\begin{equation*}
\begin{aligned}
w_{\lambda}(p_i+e^{-\frac{\lambda}{2}} z)=
~&\log{\left(\frac{16\pi e^{\lambda}}{(1+2\pi |z|^2)^2}\right)} + e^{-\lambda}Q(z_1,z_2) + 4\pi e^{-\lambda}z\nabla^2 G(p_1-p_2) z^T\\
&+ 2\lambda e^{-\lambda} + O(e^{-\frac{3\lambda}{2}}|z|^3) + O(\lambda e^{-\frac{3\lambda}{2}}|z|) + O(e^{-\lambda}).
\end{aligned}
\end{equation*}
\end{lemma}

\medskip

While for $x$ away from the blow-up points,
\begin{lemma}
\label{le4.2}
For $x\in T\setminus (B(p_1,R_0)\cup B(p_2,R_0))$, we have
\begin{equation*}
\begin{aligned}
w_{\lambda}(x)=~&-\lambda + \log{\left(4/\pi\right)} - 8\pi R(0) - 8\pi G(p_1-p_2) \\ &+8\pi\sum\limits_{i=1}^2 G(p_i-x) + 2\lambda e^{-\lambda} + O(e^{-\lambda}).
\end{aligned}
\end{equation*}
\end{lemma}

\medskip

Finally, we need to estimate the approximate solution in the neck region. It turns out that we can estimate $w_{\lambda}$ and $e^{w_{\lambda}}$ when $R_0<d(x,p_k)<2 R_0$ by comparing $w_{\lambda}$ with a function constructed by gluing the inner approximation and the outer approximation by using the cut-off function $\eta_{\lambda^{\alpha}}$ for some $\alpha\in (0,1)$. It is readily checked that the ``error" term could be controlled, please see \cite[Lemma 3.1]{gh} for more details.

\medskip

By using Lemmas \ref{le4.1}, \ref{le4.2} and \cite[Lemma 3.1]{gh} we can now estimate $e^{w_{\lambda}}$. We have
\begin{equation}
\label{4.10}
e^{w_{\lambda}}\leq \sum\limits_{i=1}^2 \frac{16\pi e^{\lambda}}{(1+2\pi e^{\lambda} (d(x,p_i))^2)^2}\left[1+\theta_{\lambda}(x)\right]
\end{equation}
where $\theta_{\lambda}$ has the property that for some constant $C>0$,
\begin{equation*}
|\theta_{\lambda}(x)|\leq C e^{-\frac{\lambda}{2}}\sum\limits_{i=1}^2 [e^{\frac{\lambda}{2}}d(x,p_i)+1].
\end{equation*}
More precisely, when $|z|\leq R_0 e^{\frac{\lambda}{2}}$, we have
\begin{equation}
\label{4.11}
\begin{aligned}
e^{w_{\lambda}(p_i+e^{-\frac{\lambda}{2}} z)} =~ &\frac{16\pi e^{\lambda}}{(1+2\pi |z|^2)^2}[1 + e^{-\lambda}Q(z_1,z_2) + 4\pi e^{-\lambda}z\nabla^2 G(p_1-p_2) z^T\\
&+ 2\lambda e^{-\lambda} + O(e^{-\frac{3\lambda}{2}}|z|^3) + O(\lambda e^{-\frac{3\lambda}{2}}|z|) + O(e^{-\lambda})].
\end{aligned}
\end{equation}
When $d(x,p_i)\geq R_0$ for $i=1,2$, we have
\begin{equation}
\label{4.12}
e^{w_{\lambda}(x)}=O(e^{-\lambda}).
\end{equation}
Finally, by exploiting (\ref{4.2}), (\ref{4.11}) and (\ref{4.12}), we can follow the same computations in \cite[Lemma 3.3]{cgh} to obtain the estimate of the error of the ansatz.
\begin{lemma} \emph{(\cite{cgh})}
\label{le4.3}
Let $S_{\rho}(u)=\Delta u + \rho \left(\frac{e^u}{\int_{T}e^u}-1\right)$. Then there exists a constant $C>0$ such that
\begin{equation*}
|S_{\rho}(w_{\lambda}(p_i+e^{-\frac{\lambda}{2}}z))|\leq C\left[\lambda e^{-\lambda} + \frac{\lambda}{(1+2\pi |z|^2)^2}+\frac{|z|^2}{(1+2\pi |z|^2)^2}\right] \quad \mathrm{for}~|z|< e^{\frac{\lambda}{2}}R_0,
\end{equation*}
and
\begin{equation*}
|S_{\rho}(w_{\lambda})(x)|\leq C \lambda e^{-\lambda} \quad \mathrm{for}~x\in T\setminus (B(p_1,R_0)\cup B(p_2,R_0)).
\end{equation*}
Furthermore, we have that $S_{\rho}(w_{\lambda})$ is evenly symmetric.
\end{lemma}

\medskip

We conclude this subsection by considering the energy of the approximate solution $w_{\lambda}$. Indeed, it is known that equation (\ref{1.1}) has a variational structure, i.e., any critical point of the energy functional
\begin{equation}
\label{4.13}
J_{\rho}(u)=\frac{1}{2}\int_{T} |\nabla u|^2 - \rho \log{\left(\int_{T}e^u\right)} + \rho \int_{T} u
\end{equation}
corresponds to a solution of (\ref{1.1}). Again by (\ref{4.2}), (\ref{4.11}) and (\ref{4.12}) and by direct computations as in \cite[Lemma 3.4]{cgh} we can obtain the following expansion on the energy of the approximate solution $w_{\lambda}$.
\begin{lemma} \emph{(\cite{cgh})}
\label{le4.4}
The energy of $w_{\lambda}$ is
\begin{equation*}
\begin{aligned}
J_{\rho}(w_{\lambda})=~&-64\pi^2 [R(0) + G(p_1-p_2)] - 16\pi\log{(2\pi)} - 16\pi - \epsilon\lambda\\
&-32\pi\lambda e^{-\lambda} - \epsilon \left[2\log{(2\pi)} - 8\pi R(0) - 8\pi G(p_1-p_2)\right] + O(e^{-\lambda}).
\end{aligned}
\end{equation*}
\end{lemma}

\medskip

\subsection{The Linearized Operator} \label{subsec.lin}
In this subsection, we shall establish the solvability theory for the linearized operator of $S_{\rho}$ under suitable orthogonality condition. Let us introduce the operator
\begin{equation}
\label{4.14}
\mathcal{L}(u)=\Delta u + \frac{\rho}{\int_{T}e^{w_{\lambda}}} e^{w_{\lambda}}u.
\end{equation}
Observe that
\begin{equation}
\label{4.15}
S'_{\rho}(w_{\lambda})(u)=\mathcal{L}\left(u-\frac{\int_{T}e^{w_{\lambda}}u}{\int_{T} e^{w_{\lambda}}}\right).
\end{equation}
Let
\begin{equation}
\label{4.16}
L(u)=e^{-\lambda}\mathcal{L}(u).
\end{equation}
If we shift the blow-up point $p_i$ to the center and rescale the torus $T$ to $T_{\lambda}$ by the factor $e^{-\frac{\lambda}{2}}$, then formally the operator $L$ converges to the operator $\tilde{L}$ in $\mathbb{R}^2$:
\begin{equation}
\label{4.17}
\tilde{L}(u)=\Delta_z u + \frac{16\pi}{(1+2\pi|z|^2)^2}u,
\end{equation}
where $z=e^{\frac{\lambda}{2}}(x-p_i)$. We point out that the operator $\tilde{L}$ can be obtained by linearizing the Liouville equation $\Delta u + e^u=0$ at the radial solution $v(z)=\log{\left(\frac{16\pi}{(1+2\pi|z|^2)^2}\right)}$. A key fact that we are going to exploit is the non-degeneracy of $v$ modulo the invariance of the Liouville equation under dilations and translations, i.e.,
\begin{equation*}
\zeta\mapsto v(z-\zeta),\quad s\mapsto v(sz)+2\log{s}.
\end{equation*}
Thus, we let
\begin{align} \label{psi}
\begin{split}
	&\psi_0(z)=\frac{\partial}{\partial s}\left[v(sz)+2\log{s}\right]\big\vert_{s=1}, \\
	&\psi_k(z)=\frac{\partial}{\partial \zeta_j}v(z-\zeta)\big\vert_{\zeta=0}\,,  \ k=1,2.
\end{split}
\end{align}
Here $\psi_k$'s coincide with $\psi_{i,k}$'s defined in Lemma \ref{le3.2}.
It is shown in \cite{bp} that the only bounded solutions of $\tilde{L}(u)=0$ in $\mathbb{R}^2$ are precisely the linear combinations of $\psi_k$, $k=0,1,2$. With a little abuse of notation, let $\psi_{i,k}:=\psi_k(e^{\frac{\lambda}{2}}(x-p_i))$ denote a function on $T_{\lambda}$ for $i=1,2$ and $k=0,1,2$.

\medskip

Next, we introduce the functional setting for the problem. To this end, we start by letting $\tilde{\eta}_R$ be the following cut-off function:
\begin{equation*}
\tilde{\eta}_R(s)=
\begin{cases}
1~&\mathrm{for}~s\leq R,\\
0~&\mathrm{for}~s\geq R+1,
\end{cases}
\quad  0\leq\tilde{\eta}_R\leq 1, \quad |\tilde{\eta}'_R(s)|\leq 2.
\end{equation*}
Let $\tilde{\eta}_{R,p}=\tilde{\eta}_R(|z-p|)$. We also write $p'_i$ to denote $e^{\frac{\lambda}{2}}p_i$. Then, we set
\begin{equation*}
L^{\infty}_e(T_{\lambda})=\Bigr\{u\in L^{\infty}(T_{\lambda}) \ \vert \ u(z)=u(-z) \Bigr\}.
\end{equation*}
We introduce the following norms
\begin{equation*}
\norm{\psi}_{\infty}=\sup\limits_{z\in T_{\lambda}}|\psi|, \qquad \norm{\psi}_{\ast}=\sup\limits_{z\in T_{\lambda}}\left(\sum\limits_{j=1}^2 (1+d(z,p'_j))^{-3}+e^{-\lambda}\right)^{-1}|\psi(z)|.
\end{equation*}
Let
\begin{equation*}
\mathcal{C}=\Bigr\{u\in L^{\infty}_e(T_{\lambda}) \ \vert \ \norm{u}_{\ast}<\infty\Bigr\},
\end{equation*}
and
\begin{equation*}
\mathcal{C}_{\ast}=\Bigr\{u\in L^{\infty}_e(T_{\lambda}) \ \vert \ u\perp \psi_{i,0}\tilde{\eta}_{R_1,p'_i},~i=1,2\Bigr\},
\end{equation*}
where $R_1>0$ is a large but fixed number. We notice here that the orthogonality condition in the definition of $\mathcal{C}_{\ast}$ is only taken with respect to the elements of the approximate kernel generated by dilations. However, it is not difficult to see that the elements in $\mathcal{C}_{\ast}$ are also perpendicular to the approximate kernel that are generated by translations, i.e.,
\begin{equation*}
u\perp \psi_{i,k}\tilde{\eta}_{R_1,p'_i}, \quad \forall i=1,2,~ k=0,1,2,\quad  u\in \mathcal{C}_{\ast}.
\end{equation*}
The main goal in this subsection is to prove a solvability result and an a priori estimate (uniform in $\lambda$) concerning the operator $L$ given in \eqref{4.16} in the functional settings defined above under suitable orthogonality conditions. To this end, let us start with the following uniform a priori estimate for an auxiliary problem with the additional orthogonality conditions of $\phi$ under translations.
\begin{lemma}
\label{le4.6}
Let $h\in \mathcal{C}\cap C^{\alpha}(T_{\lambda})$. Then, there exist $\lambda_0, C>0$ such that for any $\lambda>\lambda_0$ and any $\phi\in\mathcal{C}_{\ast}$ such that
\begin{equation}
\label{4.19}
L(\phi)=h \textrm{ in }T_{\lambda},\quad
\int_{T_{\lambda}}\tilde{\eta}_{R_1,p'_i}\psi_{i,k}\phi=0, \quad \forall i=1,2,~ k=0,1,2,
\end{equation}
it holds
\begin{equation*}
\norm{\phi}_{\infty}\leq C\norm{h}_{\ast}.
\end{equation*}
\end{lemma}
\begin{proof}
The proof follows the strategy first introduced by del Pino, Kowalczyk and Musso in \cite{dkm} in dealing with a singularly perturbed Liouville-type equation on a bounded domain with Dirichlet boundary condition. The argument was then suitably adapted to the flat torus case in \cite{cgh} so we will state just the main steps referring to \cite[Lemma 4.2]{cgh} for a detailed proof.

\medskip

The proof is obtained by contradiction assuming that there exist sequences $\lambda_n\to+\infty$, $h_n$ with $\norm{h_n}_{\ast}\to0$ and $\phi_n$ with $\norm{\phi_n}_{\infty}=1$ satisfying \eqref{4.19}. The contradiction is obtained after the following steps.

\medskip

\noindent\textbf{Step 1.}
The first step is to construct a positive supersolution $V$ in order to show that the operator $L$ satisfies the maximum principle on the torus outside the bubbling disks $\tilde T_\lambda= T_\lambda \setminus \cup_{i=1}^2 B(p_i',R_2')$, for $R_2'>0$ sufficiently large, i.e. if $L(u)\leq0$ in $\tilde T_\lambda$ and $u\geq0$ on $\partial \tilde T_\lambda$, then $u\geq0$ in $\tilde T_\lambda$. This is done by defining a suitable projection of the radial solution $f_0(r)=\frac{r^2-1}{r^2+1}$ in $\mathbb{R}^2$ of
$$
	\Delta f_0+\frac{8}{(1+r^2)^2}f_0=0,
$$
to a function space on $T_\lambda$ in order to satisfy the periodic boundary conditions.

\medskip

\noindent\textbf{Step 2.}
The second step is to prove that there exists a constant $C>0$ such that if $L(\phi)=h$ in $T_\lambda$, then
$$
\norm{\phi}_{\infty}\leq C[\norm{\phi}_{in}+\norm{h}_{\ast}],
$$
where $\norm{\cdot}_{in}$ denotes the ``inner" norm of a function on $T_{\lambda}$ in the bubbling disks, i.e.
$$
 \norm{\phi}_{in}=\sup_{\cup_{j=1}^2 B(p_j',R_2')} |\phi|.
$$
One can use suitable barrier functions in $\tilde T_\lambda$ jointly with the maximum principle of Step 1 to derive the above claim.

\medskip

\noindent\textbf{Step 3.}
In the final step we will employ a convergence argument to finally deduce a contradiction. By assumptions and by Step 2 we get $\norm{\phi_n}_{in}\geq\delta>0$. Therefore, one can see that $\phi_n$ in a bubbling disk locally converge to a bounded non-zero solution of
$$
\tilde{L}(\hat\phi)=0
$$
given in \eqref{4.17}. Hence, $\hat\phi$ is a linear combination of $\psi_k$, $k=0,1,2$, defined in \eqref{psi}. On the other hand, the orthogonal conditions on $\phi_n$ imply $\hat\phi\equiv0$, yielding a contradiction.
\end{proof}

\medskip

We can now prove the solvability and a priori estimate of the following problem \eqref{4.18}.
\begin{proposition}\label{pr4.5}
Let $h\in\mathcal{C}$. Then, there exist $\lambda_0, C>0$, such that for all $\lambda>\lambda_0$, there exist a unique $\phi\in\mathcal{C}_{\ast}$ and numbers $c_i$, $i=1,2$ such that
\begin{equation}
\label{4.18}
L(\phi)=h+\sum\limits_{i=1}^2 c_i \tilde{\eta}_{R_1,p'_i}\psi_{i,0}~\textrm{ in }~ T_{\lambda}.
\end{equation}
Moreover, if $h\in C^{\alpha}(T_{\lambda})$, then
\begin{equation} \label{est}
\norm{\phi}_{\infty}\leq C\norm{h}_{\ast}.
\end{equation}
\end{proposition}

\begin{proof}
We start by proving the a priori estimate \eqref{est}. One can apply Lemma \ref{le4.6} to get
$$
	\norm{\phi}_{\infty}\leq C\left[\norm{h}_{\ast}+\sum_{i=1}^2 |c_i|\right].
$$
We can reason exactly as in \cite[Proposition 4.1]{cgh} and after multiplying the equation \eqref{4.18} by the test function $\tilde{\eta}_{R_3',p'_i}\psi_{i,0}$, $R_3'>0$ sufficiently large, derive $|c_i|\leq C\norm{h}_*$. Thus, \eqref{est} holds true.

\medskip

The existence of a solution to \eqref{4.18} follows from the Fredholm alternative since we know that equation \eqref{4.18} has a unique solution if and only if the associated homogeneous problem (i.e. with $h\equiv0$) has only the trivial solution. By the a priori estimate we conclude that this is the case and the proof is concluded.
\end{proof}

\medskip

Finally, we can reason as in the proof of Proposition \ref{pr4.5} and deduce the following main result of this subsection, in which one more orthogonal condition to $\phi$ is imposed, see \cite[Corollary~4.3]{cgh}.
\begin{proposition} \emph{(\cite{cgh})}
\label{cr4.7}
Let $h\in\mathcal{C}$. Then, there exist $\lambda_0, C>0$, such that for all $\lambda>\lambda_0$, there exist a unique $\phi\in\mathcal{C}_{\ast}$ and numbers $c_i$, $i=0,1,2$ such that
\begin{equation}
\label{4.20}
L(\phi)=h+\sum\limits_{i=1}^2 c_i \tilde{\eta}_{R_1,p'_i}\psi_{i,0}+c_0 \textrm{ in }T_{\lambda},\quad\phi\perp e^{w_{\lambda}}.
\end{equation}
Moreover, if $h\in C^{\alpha}(T_{\lambda})$, then
\begin{equation*}
\norm{\phi}_{\infty}\leq C \norm{h}_{\ast}.
\end{equation*}
\end{proposition}

\medskip

By means of the latter result we can define a continuous linear map $T:\mathcal{C}\to L^{\infty}_e(T_{\lambda})$ given by
$$
	h\mapsto T(h):=\phi,
$$
where $\phi$ is the unique solution of problem (\ref{4.20}) obtained in Proposition \ref{cr4.7}.

\medskip

\subsection{Finite-dimensional reduction and proof of Theorem \ref{th1}}
\noindent We are ready to reduce the infinite dimensional problem of finding $\phi$ such that
\begin{equation}
\label{4.21}
S_{\rho}(w_{\lambda}+\phi)=0
\end{equation}
to a one-dimensional problem of finding appropriate scale $\lambda$ with given $\rho$. To this end we first expand $S_{\rho}(w_{\lambda}+\phi)$ as
\begin{equation}
\label{4.22}
S_{\rho}(w_{\lambda}+\phi)=S_{\rho}(w_{\lambda})+\mathcal{L}\left(\phi-\frac{\int_{T}e^{w_{\lambda}}\phi}{\int_{T}e^{w_{\lambda}}}\right) + N(\phi),
\end{equation}
where
\begin{equation}
\label{4.23}
N(\phi)=\left[\frac{\rho}{\int_{T}e^{w_{\lambda}+\phi}}e^{\phi}-\frac{\rho}{\int_{T}e^{w_{\lambda}}}-\left(\phi-\frac{\int_{T}e^{w_{\lambda}}\phi}{\int_{T}e^{w_{\lambda}}}\right)\right]e^{w_{\lambda}}.
\end{equation}
Since $S_{\rho}(w_{\lambda}+\phi)$ is invariant under adding a constant to $\phi$, we can further assume that
\begin{equation*}
\int_{T} e^{w_{\lambda}}\phi=0.
\end{equation*}
By slightly abuse of notation we still denote $\phi$ as a function in $\mathcal{C}_{\ast}$. Moreover, we consider problem (\ref{4.21}) in the dilated coordinates, i.e. $w_{\lambda}$, $S_{\rho}(w_{\lambda})$ and $N(\phi)$ are now treated as functions on $T_{\lambda}$.

\medskip

In order to obtain a solution to \eqref{4.21} we first exploit the solvability of the linearized operator established in subsection \ref{subsec.lin} to solve the following intermediate problem.
\begin{lemma} \label{lem.inter}
There exist $\lambda_0, C>0$, such that for all $\lambda>\lambda_0$, there exist a unique $\phi\in\mathcal{C}_{\ast}$ and numbers $c_i$, $i=0,1,2$ such that
\begin{equation}
\label{4.24}
L(\phi)=-e^{-\lambda}\left[S_{\rho}(w_{\lambda})+N(\phi)\right]+\sum\limits_{i=1}^2 c_i \tilde{\eta}_{R_1,p'_i}\psi_{i,0}+c_0 \textrm{ in }T_{\lambda}, \quad \phi\perp e^{w_{\lambda}}.
\end{equation}
Moreover, it holds
\begin{equation}
\label{4.25}
\norm{\phi}_{\infty}\leq C e^{-\frac{\lambda}{2}}.
\end{equation}
\end{lemma}

\begin{proof}
Let $T:\mathcal{C}\to L^{\infty}_e(T_{\lambda})$ be the continuous linear map defined after Proposition~\ref{cr4.7}. Then, we rewrite \eqref{4.24} as
$$
	\phi=A(\phi):=T\left( -e^{-\lambda}\left[S_{\rho}(w_{\lambda})+N(\phi)\right] \right)
$$
in the subspace $\mathcal{F}=\{ \phi\in\mathcal{C}_{\ast} \ \vert \ \phi\perp e^{w_\lambda}, \, \norm{\phi}_{\infty}\leq C e^{-\frac{\lambda}{2}} \}$. By using Proposition~\ref{cr4.7} and the estimate of the error $S_{\rho}(w_{\lambda})$ in Lemma \ref{le4.3} it is not difficult to show that the operator $A$ is a contraction map in the space $\mathcal{F}$. Thus, there exists a unique fixed point which is a solution to \eqref{4.24}.
\end{proof}

\medskip

To conclude and get a solution to \eqref{4.21} we are left with showing that $c_0=c_1=c_2=0$ in \eqref{4.24}. To this end, we have the following properties.
\begin{lemma} \label{lem.propr}
Let $\phi$ and $c_i$, $i=0,1,2$ be given as in Lemma \ref{lem.inter}. Then, it holds:
\begin{itemize}
\item[(1)] $c_1=c_2$ and $c_0 = -2 e^{-\lambda}\mathcal{A}c_1$, where $\mathcal{A}=\int_{\mathbb{R}^2}\tilde{\eta}_{R_1}\psi_0(z)dz$.

\item[(2)] $\norm{\frac{\partial \phi}{\partial \lambda}}_{\infty}\leq C e^{-\frac{\lambda}{2}}$.
\end{itemize}
\end{lemma}

\begin{proof}
(1) From the invariance of problem (\ref{4.24}) by the change $z\mapsto z+p'_2$ we have
$$
 \left\langle L(\phi), \tilde{\eta}_{R_1,p'_1}\psi_{1,0} \right\rangle=	\left\langle L(\phi), \tilde{\eta}_{R_1,p'_2}\psi_{2,0} \right\rangle.
$$
Moreover, integrating the equation in (\ref{4.24}) on $T_{\lambda}$, we have
\begin{equation}
\label{4.26}
 \mathcal{A}(c_1+c_2)+e^{\lambda}c_0=0
\end{equation}
and (1) holds true.

\medskip

\noindent (2) The idea is to differentiate \eqref{4.24} with respect to $\lambda$, write
$$
	L\left(\frac{\partial \phi}{\partial \lambda}\right)=\tilde h+\sum\limits_{i=1}^2 \tilde c_i \tilde{\eta}_{R_1,p'_i}\psi_{i,0},
$$
for some suitable $\tilde h$, $\tilde c_i$ and finally exploit Proposition \ref{pr4.5}. We can follow the same computations as in \cite[Lemma 5.3]{cgh} so we omit the details.
\end{proof}

\medskip

Moreover, we have the following property concerning the energy functional $J_\rho$ defined in \eqref{4.13}.
\begin{lemma}
\label{le4.8}
Let $\phi$ be given as in Lemma \ref{lem.inter}. Then, $J_{\rho}(w_{\lambda}+\phi)$ is a $C^1$ function with respect to $\lambda$ for $\lambda\in(\lambda_1,\lambda_2)$ and hence it has a local maximum point $\lambda_{\ast}$. Furthermore, we have
\begin{equation*}
\epsilon=\left(32\pi+o(1)\right)\lambda e^{-\lambda},
\end{equation*}
as $\epsilon\rightarrow 0$.
\end{lemma}

\begin{proof}
We first make the following expansion
$$
J_{\rho}(w_{\lambda}+\phi)=J_{\rho}(w_{\lambda})+\langle S_{\rho}(w_{\lambda}+\theta\phi),\phi  \rangle,
$$
for some $\theta\in(0,1)$. Then,
$$
	S_{\rho}(w_{\lambda}+\theta\phi)=S_{\rho}(w_{\lambda})+\theta\Delta\phi+O(e^{-\frac{\lambda}{2}}e^{w_\lambda})
$$
Exploiting $\norm{\phi}_{\infty}\leq Ce^{-\frac{\lambda}{2}}$, the estimate of the error $S_{\rho}(w_{\lambda})$ in Lemma \ref{le4.3} and reasoning as in \cite[Lemma 6.1]{cgh} for the term $\Delta\phi$ it is easy to show that
$$
J_{\rho}(w_{\lambda}+\phi)=J_{\rho}(w_{\lambda})+O(e^{-\lambda}).
$$
Then, by letter estimate and by Lemma \ref{le4.4} we have
\begin{equation*}
\begin{aligned}
J_{\rho}(w_{\lambda})=~&-64\pi^2 [R(0) + G(p_1-p_2)] - 16\pi\log{(2\pi)} - 16\pi - \epsilon\lambda\\
&-32\pi\lambda e^{-\lambda} - \epsilon \left[2\log{(2\pi)} - 8\pi R(0) - 8\pi G(p_1-p_2)\right] + O(e^{-\lambda}).
\end{aligned}
\end{equation*}
The proof of Lemma \ref{le4.8} follows then easily.
\end{proof}

\medskip

Finally, we can prove now the main result of this section. It turns out that the scale $\lambda_*$ given by the Lemma \ref{le4.8} is the right choice to solve the problem \eqref{4.21}.

\medskip

\begin{proof}[Proof of Theorem \ref{th1}.]
We have to prove that for $\lambda=\lambda_{\ast}$, $c_0=c_1=c_2=0$ in (\ref{4.24}). Since $\lambda_*$ is a critical point of $J_\rho(w_\lambda+\phi)$ we have
$$
	0=\frac{\partial J_\rho(w_\lambda+\phi)}{\partial\lambda}\vert_{\lambda=\lambda_*}=\left\langle S_\rho(w_\lambda+\phi),\frac{\partial(w_\lambda+\phi)}{\partial\lambda} \right\rangle\vert_{\lambda=\lambda_*}.
$$
On the other hand, by using Lemma \ref{lem.propr} and by direct computations as in \cite[Lemma 6.3]{cgh} it is not difficult to get
$$
	\left\langle S_\rho(w_\lambda+\phi),\frac{\partial(w_\lambda+\phi)}{\partial\lambda} \right\rangle= \left(Ce^{\frac{\lambda}{2}}+O(1)\right)c_1
$$
and hence it follows $c_0=c_1=c_2=0$. Thus,
$$
S_{\rho}(w_{\lambda_*}+\phi_*)=0
$$
and $w_{\lambda_*}+\phi_*$ is the desired blowing-up solution to (\ref{1.1}).
\end{proof}


\vspace{0.25cm}
\begin{thebibliography}{99}
\bibitem{bp} Sami Baraket, and Frank Pacard. Construction of singular limits for a semilinear elliptic equation in dimension 2.
\emph{Calc. Var. \& P.D.E.}, 6(1):1-38, 1997.

\bibitem{B1} D.Bartolucci. Global bifurcation analysis of mean field equations and
the Onsager microcanonical description of two-dimensional turbulence.  \emph{Calc. Var. \& P.D.E.}, 58(1) (2019), 58:18;
DOI:10.1007/s00526-018-1445-4.

\bibitem{bcct} D. Bartolucci, C.C. Chen, C.S. Lin, and G. Tarantello.
{Profile of Blow Up Solutions To Mean Field Equations with Singular Data},
\emph{Comm. in P. D. E.}, {29}(7-8):1241-1265, 2004.

\bibitem{bgjm}
Daniele Bartolucci, Changfeng Gui, Aleks Jevnikar, and Amir Moradifam.
A singular Sphere Covering Inequality: uniqueness and symmetry of solutions to singular Liouville-type equations,
\emph{Math. Ann.}, to appear; doi: 10.1007/s00208-018-1761-1.

\bibitem{bjl} Daniele Bartolucci, Aleks Jevnikar, and Chang-Shou Lin.
Non-degeneracy and uniqueness of solutions to singular mean field equations on bounded domains,
\emph{J. Diff. Eq.}, 266(1) (2019), 716-741; DOI 10.1016/j.jde.2018.07.053.


\bibitem{bjly} Daniele Bartolucci, Aleks Jevnikar, Youngae Lee, and Wen Yang.
Uniqueness of bubbling solutions of mean field equations. {\em J. Math. Pures Appl.}, 123 (2019), 78-126.

\bibitem{bjly2} Daniele Bartolucci, Aleks Jevnikar, Youngae Lee, and Wen Yang.
Non degeneracy, Mean Field Equations and the Onsager theory of 2D turbulence. {\em Arch. Rat. Mech. Anal.}, 230(1):397-426, 2018.

\bibitem{bjly3} Daniele Bartolucci, Aleks Jevnikar, Youngae Lee, and Wen Yang. Local uniqueness of $m$-bubbling sequences for the Gel'fand equation. \emph{Comm. in P. D. E.}, to appear; doi: 10.1080/03605302.2019.1581801.

\bibitem{BDeM}
Daniele Bartolucci, and Francesca De Marchis. On the Ambjorn-Olesen electroweak condensates.
\emph{{Jour. Math. Phys.}} { 53}, 073704, 2012; doi: 10.1063/1.4731239.

\bibitem{BDeM2} Daniele Bartolucci, and Francesca De Marchis.
{Supercritical Mean Field Equations on convex domains and the Onsager's
statistical description of two-dimensional turbulence}, {\em Arch. Rat. Mech. Anal.}, 217(2):525-570, 2015.
DOI: 10.1007/s00205-014-0836-8.

\bibitem{BdMM} Daniele Bartolucci, Francesca De Marchis, and Andrea Malchiodi. {Supercritical conformal metrics on
surfaces with conical singularities}. \emph{Int. Math. Res. Not.} 24:5625-5643, 2011; DOI: 10.1093/imrn/rnq285.

\bibitem{bl} D. Bartolucci, and C.S. Lin. {Uniqueness Results for Mean Field Equations with Singular Data},
\emph{Comm. in P. D. E.}, {34}(7):676-702, 2009.

\bibitem{BLin3} D. Bartolucci, and C.S. Lin, {Existence and uniqueness for
Mean Field Equations on multiply connected domains at the critical parameter},
{\emph{Math. Ann.}}, {359}:1-44, 2014; DOI 10.1007/s00208-013-0990-6.

\bibitem{BLT} D. Bartolucci, C.S. Lin, and G. Tarantello, {Uniqueness and symmetry results for
solutions of a mean field equation on ${\mathbb{S}}^{2}$ via a new bubbling phenomenon},
\emph{{Comm. Pure Appl. Math.}}, {64}(12):1677-1730, 2011.

\bibitem{BMal} Daniele Bartolucci, and Andrea Malchiodi. { An improved geometric
inequality via vanishing moments, with applications to singular
Liouville equations}. \emph{{Comm. Math. Phys.}} {322}:415-452, 2013.

\bibitem{bt} Daniele Bartolucci, and Gabriella Tarantello. {Liouville type equations with
singular data and their applications to periodic multivortices for the
electroweak theory}. \emph{Comm. Math. Phys.}, {229}:3-47, 2002.

\bibitem{bt2} Daniele Bartolucci, and Gabriella Tarantello. {Asymptotic blow-up analysis for singular Liouville type equations with
applications}, \emph{J. Differential Equations}, {262}(7):3887-3931, 2017.

\bibitem{bgp} Luca Battaglia, Massimo Grossi, and Angela Pistoia. Non-uniqueness of blowing-up solutions to the Gelfand problem. arXiv:1902.03484, 2019.

\bibitem{bm}
H. Brezis, and F. Merle.
{Uniform estimates and blow-up behaviour for
solutions of $-\Delta u = V(x)e^{u}$ in two dimensions},
{\emph{Comm. in P.D.E.},} {16}(8,9):1223-1253, 1991.

\bibitem{clmp} Emanuele Caglioti, Pierre-Louis. Lions, Carlo Marchioro, and Mario Pulvirenti. A special class of stationary flows for two-dimensional euler equations: A statistical mechanics description. {\em Communications in Mathematical Physics}, 143(3):501-525, 1992.

\bibitem{clmp2} Emanuele Caglioti, Pierre-Louis. Lions, Carlo Marchioro, and Mario Pulvirenti. A special class of stationary flows for two-dimensional euler equations: A statistical mechanics description. Part II. {\em Communications in Mathematical Physics}, 174(2):229-260, 1995.

\bibitem{cll} Daomin Cao, Shuanglong Li, and Peng Luo. Uniqueness of positive bound states with multi-bump for nonlinear Schrödinger equations. {\em Calculus of Variations and Partial Differential Equations}, 54(4):4037-4063, 2015.

\bibitem{cny} Daomin Cao, Ezzat S. Noussair, and Shusen Yan. Existence and uniqueness results on single peaked solutions of a semilinear problem. {\em Annales de l'Institut Henri Poincar\'e. Analyse Non Lin\'eaire,} 15(1):73-111, 1998.

\bibitem{cama} Alessandro Carlotto, and Andrea Malchiodi. {
Weighted barycentric sets and singular Liouville equations on compact surfaces},
\emph{J. Funct. Anal.}, 262(2):409-450, 2012.

\bibitem{cfl} Hsungrow Chan, Chun-Chieh Fu, and Chang-Shou Lin. Non-topological multi-vortex solutions to the self-dual chern-simons-higgs equation. {\em Communications in Mathematical Physics,} 231(2):189-221, 2002.

\bibitem{CCL} S.Y.A. Chang, C.C. Chen, and C.S. Lin, {Extremal functions for a mean field equation in two dimension},
Lecture on Partial Differential Equations, New Stud. Adv. Math. {2} Int. Press, Somerville, MA, 61-93, 2003.

\bibitem{ckcmp} Sagun Chanillo, and Michael Kiessling. Rotational symmetry of solutions of some nonlinear problems in statistical mechanics and in geometry. {\em Communications in Mathematical Physics,} 160(2):217-238, 1994.


\bibitem{cl} Chiun-Chuan Chen, and Chang-Shou Lin. Sharp estimates for solutions of multi-bubbles in compact riemann surfaces. {\em Communications on Pure and Applied Mathematics,} 55(6):728-771, 2002.

\bibitem{cl2} Chiun-Chuan Chen, and Chang-Shou Lin. Topological degree for a mean field equation on riemann surfaces. {\em Communications on Pure and Applied Mathematics,} 56(12):1667-1727, 2003.

\bibitem{cl3} Chiun-Chuan Chen, and Chang-Shou Lin. Mean field equation of liouville type with singular data: Topological degree. {\em Communications on Pure and Applied Mathematics,} 68(6):887-947, 2015.

\bibitem{cl4} Chiun-Chuan Chen, and Chang-Shou Lin. {Mean field equations of Liouville type with singular data: shaper estimates.}
 \emph{Discrete Contin.  Dyn.  Syst.}, 28(3):1237-1272, 2010.

\bibitem{clw} Chiun-Chuan Chen, Chang-Shou Lin, and Guofang Wang. Concentration phenomena of two-vortex solutions in a chern-simons model. {\em Annali della Scuola Normale Superiore di Pisa - Classe di Scienze,} 3(2):367-397, 2004.

\bibitem{cgs} Robin Ming Chen, Yujin Guo, and Daniel Spirn. Asymptotic behavior and symmetry of condensate solutions in electroweak theory. {\em Journal d'Analyse Math\'ematique,} 117(1):47-85, 2012.

\bibitem{cli} Wenxiong Chen and Congming Li. Classification of solutions of some nonlinear elliptic equations. {\em Duke Mathematical Journal,} 63(3):615-622, 1991.

\bibitem{co} Xinfu Chen, and Yoshihito Oshita. An application of the modular function in nonlocal variational problems. {\em Archive for Rational Mechanics and Analysis,} 186(1):109-132, 2007.

\bibitem{cklw} Zhijie Chen, Ting-Jung Kuo, Chang-Shou Lin, and Chin-Lung Wang. Green function, Painlev\'{e} VI equation, and Eisenstein series of wight one. {\em Journal of Differential Geometry,} 108(2):185-241, 2018.

\bibitem{cgh} Ze Cheng, Changfeng Gui, and Yeyao Hu. Blow-up solutions for a mean field equation on a flat torus. \emph{Indiana University Math Journal}, to appear.

\bibitem{ckhp} Kwangseok Choe, and Namkwon Kim. Blow-up solutions of the self-dual chern–simons-higgs vortex equation. {\em Annales de l'Institut Henri Poincar\'e Non Lin\'ear Analysis,} 25(2):313-338, 2008.

\bibitem{dkm} Manuel del Pino, Michal Kowalczyk, and Monica Musso. Singular limits in liouville-type equations. {\em Calculus of Variations and Partial Differential Equations,} 24(1):47-81, 2005.

\bibitem{dly} Yinbin Deng, Chang-Shou Lin, and Shusen Yan. On the prescribed scalar curvature problem in $\mathbb{R}^n$, local uniqueness and periodicity. {\em J. Math. Pures Appl.}, 104(6):1013 - 1044, 2015.

\bibitem{DJLW} W. Ding, J. Jost, J. Li, and G. Wang. {Existence results for
mean field equations},  \emph{Ann. Inst. H. Poincar\'e Anal. Non Lin\'eaire}, {16}:653-666, 1999.

\bibitem{dj} Z. Djadli. {Existence result for the mean field problem
on Riemann surfaces of all genuses}, \emph{Comm. Contemp. Math.}  10(2):205-220, 2008.

\bibitem{deg} Jean Dolbeault, Maria J. Esteban, and Gabriella Tarantello. Multiplicity results for the assigned gauss curvature problem in $\mathbb{R}^2$ . {\em Nonlinear Analysis: Theory, Methods and Applications,} 70(8):2870 - 2881, 2009.

\bibitem{ef} Pierpaolo Esposito, and Pablo Figueroa. Singular mean field equations on compact Riemann surfaces. \emph{Nonlinear Analysis: Theory, Methods \& Applications}, 111:33-65, 2014.

\bibitem{g} Massimo Grossi. On the number of single-peak solutions of the nonlinear Schr\"odinger equation. {\em Annales de l'Institut Henri Poincar\'e Non Lin\'ear Analysis}, 19(3):261-280, 2002.

\bibitem{gh} Changfeng Gui, and Yeyao Hu. Non-axially symmetric solutions of a mean field equation on $\mathbb{S}^2$. arXiv:1709.02474, 2017.

\bibitem{gm} Changfeng Gui, and Amir Moradifam. The sphere covering inequality and its applications. \emph{Invent. Math.} (2018); https://doi.org/10.1007/s00222-018-0820-2.

\bibitem{GM2} Changfeng Gui, and Amir Moradifam. Symmetry of solutions of a mean field equation on flat tori. \emph{Int. Math. Res. Not.} (2016); https://doi.org/10.1093/imrn/rnx121.

\bibitem{GM3}{ Changfeng Gui, and Amir Moradifam.} {Uniqueness of solutions of mean field equations in $\mathbb{R}^2$}. \emph{Proc. Amer. Math. Soc.}, 146(3):1231-1242, 2018.

\bibitem{glw} Yujin Guo, Changshou Lin, and Juncheng Wei. Local uniqueness and refined spike profiles of ground states for two-dimensional attractive Bose-Einstein condensates. {\em SIAM J. Math. Anal.}, 49(5):3671-3715, 2017.

\bibitem{gpy} Yuxia Guo, Shuangjie Peng, and Shusen Yan. Local uniqueness and periodicity induced by concentration. {\em Proceedings of the London Mathematical Society,} 114(6):1005-1043, 2017.

\bibitem{hkp} Jooyoo Hong, Yoonbai Kim, and Pong Youl Pac. Multivortex solutions of the abelian chern-simons-higgs theory. {\em Physical Review Letters,} 64(19):2230, 1990.

\bibitem{jw} Roman Jackiw, and Erick J. Weinberg. Self-dual chern-simons vortices. {\em Physical Review Letters}, 64(19):2234, 1990.

\bibitem{KLB} J. Katz, D Lynden-Bell, The Gravothermal instability in two dimensions,\emph{M.N.R.A.S. } \textbf{184} (1978) 709-712.

\bibitem{kw} J. L. Kazdan, and F. W. Warner. Curvature functions for compact 2-manifolds. \emph{Ann. Math.}, 99:14-74, 1974.

\bibitem{k} Michael K.-H. Kiessling. Statistical mechanics of classical particles with logarithmic interactions. {\em Communications on Pure and Applied Mathematics,} 46(1):27-56, 1993.

\bibitem{l} YanYan Li. Harnack type inequality: the method of moving planes. {\em Communications in Mathematical Physics,} 200(2):421-444, 1999.

\bibitem{ls} Y.Y. Li, and I. Shafrir. {Blow-up analysis for Solutions of $-\Delta u = V(x)e^{u}$
in dimension two}, \emph{{Ind. Univ. Math. J.}},  {43}(4):1255-1270, 1994.

\bibitem{lin} Chang-Shou Lin. Topological degree for mean field equations on $\mathbb{S}^2$ . {\em Duke Mathematical Journal,} 104(3):501-536, 2000.

\bibitem{lin2} Chang-Shou Lin. Uniqueness of solutions to the mean field equations for the spherical Onsager vortex. {\em Archive for Rational Mechanics and Analysis,} 153(2):153-176, 2000.

\bibitem{ly} Chang-shou Lin, and Shusen Yan. On the mean field type bubbling solutions for Chern-Simons-Higgs equation. {\em Advances in Mathematics,} 338:1141-1188, 2018.

\bibitem{mw} Li Ma, and Juncheng Wei. Convergence for a liouville equation. {\em Commentarii Mathematici Helvetici,} 76(3):506-514, 2001.

\bibitem{pt} Arkady Poliakovsky, and Gabriella Tarantello. On a planar liouville-type problem in the study of selfgravitating strings. {\em Journal of Differential Equations,} 252(5):3668-3693, 2012.

\bibitem{sstw} Yuguang Shi, Jiacheng Sun, Gang Tian, and Dongyi Wei. Uniqueness of the mean field equation and rigidity of hawking mass. arXiv:1706.06766, 2017.

\bibitem{suz} T. Suzuki. {Global analysis for a two-dimensional elliptic eiqenvalue problem with the exponential
                nonlinearity}. \emph{Ann. Inst. H. Poincar\'e Anal. Non Lin\'eaire}, {9}(4):367-398, 1992.

\bibitem{t} Gabriella Tarantello. Multiple condensate solutions for the Chern-Simons-Higgs theory. {\em Journal of Mathematical Physics,} 37(8):3769-3796, 1996.

\bibitem{t2} Gabriella Tarantello. Analytical, geometrical and topological aspects of a class of mean field
equations on surfaces. {\em Discrete and Continuous Dynamical Systems,} 28(3):931-973, 2010.


\bibitem{t3} Gabriella Tarantello.\emph{Jour. Funct. An.} {Blow-up analysis for a cosmic strings equation}, 272: 255–338, 2017.

\bibitem{tr}
M. Troyanov. Prescribing curvature on compact surfaces with conical singularities. \emph{Trans. Amer. Math. Soc.}, 324:793-821, 1991.

\bibitem{w} Juncheng Wei. On single interior spike solutions of the gierermeinhardt system: uniqueness and spectrum estimates. {\em European Journal of Applied Mathematics,} 10(4):353-378, 1999.

\bibitem{w2} Juncheng Wei. Uniqueness and critical spectrum of boundary spike solutions. {\em Proceedings
of the Royal Society of Edinburgh. Section A. Mathematics,} 131(6):1457-1480, 2001.

\bibitem{wo} G. Wolansky, {On steady distributions of self-attracting
clusters under friction and fluctuations}, \emph{Arch. Rational Mech. An.} ,119: 355-391, 1992.

\bibitem{y} Yisong Yang. Self-duality of the gauge field equations and the cosmological constant. {\em Communications in Mathematical Physics,} 162(3):481-498, 1994.
\end{thebibliography}
\end{document}